\documentclass[10pt, leqno]{article}

\usepackage{amsmath,amssymb,amsthm, epsfig}

\usepackage{hyperref}

\usepackage{amsmath}

\usepackage{amssymb}

\usepackage{hyperref}

\usepackage{color}

\usepackage{ulem}

\usepackage{epsfig}

\usepackage{dsfont}

\usepackage{mathrsfs}

\usepackage{wasysym}

\title{On the H\" older regularity for solutions of integro-differential equations like the anisotropic fractional Laplacian}

\author{\it by \smallskip \\
E. B. dos Santos \quad $\&$ \quad R. Leit\~ao 
    \footnote{ \noindent \textsc{dos Santos}.
Universidade Federal do Cear\'{a} - UFC. Department of Mathematics. Fortaleza - CE, Brazil - 60455-760.
\texttt{E-mail address: elizaphanbraga@yahoo.com.br}
    }    
\footnote{
\noindent \textsc{R. Leit\~ao}.
Universidade Federal do Cear\'{a} - UFC. Department of Mathematics. Fortaleza - CE, Brazil - 60455-760.
\texttt{E-mail address: rleitao@mat.ufc.br}
}        
                                                                                        }

\newlength{\hchng}
\newlength{\vchng}
\def \dist {\mathrm{dist}}

\newtheorem{theorem}{Theorem}[section]
\newtheorem{lemma}[theorem]{Lemma}

\newtheorem{corollary}[theorem]{Corollary}
\theoremstyle{definition}

\newtheorem{definition}[theorem]{Definition}

\theoremstyle{remark}
\newtheorem{remark}[theorem]{Remark}

\numberwithin{equation}{section}

\numberwithin{equation}{section}

\newcommand{\intav}[1]{\mathchoice {\mathop{\vrule width 6pt height 3 pt depth  -2.5pt
\kern -8pt \intop}\nolimits_{\kern -6pt#1}} {\mathop{\vrule width
5pt height 3  pt depth -2.6pt \kern -6pt \intop}\nolimits_{#1}}
{\mathop{\vrule width 5pt height 3 pt depth -2.6pt \kern -6pt
\intop}\nolimits_{#1}} {\mathop{\vrule width 5pt height 3 pt depth
-2.6pt \kern -6pt \intop}\nolimits_{#1}}}

\begin{document}
\maketitle

\begin{abstract}
In this paper we study integro-differential equations like the anisotropic fractional Laplacian. As in [Silvestre, Indiana Univ. Math. J. 55, 2006], we adapt the De Giorgi technique to achieve the $C^{\gamma}$-regularity for solutions of class $C^{2}$ and use the geometry found in [Caffarelli, Leitão, and Urbano, Math. Ann. 360, 2014] to get an ABP estimate, a Harnack inequality and the interior $C^{1, \gamma}$ regularity for viscosity solutions.
\bigskip


 \noindent{\sc AMS Subject Classification MSC 2010}: 26A33; 35J70; 47G20, 35J60, 35D35, 35D40, 35B65

\tableofcontents

\end{abstract} 

\section{Introduction}

In \cite{LEITAO}, the second author presents the anisotropic fractional Laplacian

\begin{eqnarray}
\label{anisotropic fractional Laplacian}
(- \Delta)^{\beta, s} f(x) = C_{\beta, s} \int_{\mathbb{R}^{n}} \dfrac{f(x) - f(\varsigma)}{\left( \sum_{i=1}^{n} \vert \varsigma_{i} - x_{i} \vert^{b_{i}} \right)^{\frac{c+s}{2}}} d\varsigma,
\end{eqnarray}
where $\beta = \left( b_{1}, \dots, b_{n} \right) \in \mathbb{R}^{n}$ represents the different homogeneities in different directions, $b_{i} > 0$, $0 < s < 2$, $c = \displaystyle \sum_{i=1}^{n} \frac{2}{b_{i}}$ and $C_{\beta, s} > 0$ is a normalization constant. In this work we develop a regularity theory for integro-differential equations like the anisotropic fractional Laplacian
\begin{equation}\label{MAIN EQUATION}
L_{}u\left( x \right) = 0,
\end{equation}
where
\begin{equation}\label{operator}
L_{}u\left( x \right) :=  \int_{\mathbb{R}^{n}}\left( u\left( x + y \right) - u\left(x \right) - \chi_{B_{1}}\left( y \right) \nabla u \left( x \right)\cdot y \right) \mathcal{K}_{}\left( y\right)dy,
\end{equation}
$0 < s < \frac{4}{b_{\max}}$, and the kernel $\mathcal{K}_{}$ is symmetric, $\mathcal{K}_{}(y)=\mathcal{K}_{}(-y)$, and satisfy the anisotropic bounds
\begin{equation}\label{Kernel cond 2}  
\dfrac{\lambda \ q_{\max, s} }{\Vert y \Vert^{c+s} } \leq \mathcal{K}_{}\left( y \right) \leq   \dfrac{\Lambda \ q_{\max, s} }{\Vert y \Vert^{c+s}}, \quad \forall y \in \mathbb{R}^{n} \setminus \left\lbrace 0 \right\rbrace,
\end{equation}
where $0< \lambda \leq \Lambda$ and we denote $b_{\max}=\max\left\lbrace b_{1}, \dots, b_{n} \right\rbrace$, 
$$\Vert y \Vert^{2} = \sum_{i=1}^{n} \vert y_{i} \vert^{b_{i}} \quad \text{and} \quad q_{\max,s}=\dfrac{4}{b_{\max}} - s.$$

Integro-differential equations appear in the context of discontinuous stochastic processes. For example, competitive stochastic games with two or more players, which are allowed to choose from different strategies at every step in order to maximize the expected value of some function at the first exit point of a domain. Integral operators like \eqref{anisotropic fractional Laplacian} correspond to purely jump processes when diffusion and drift are neglected. The anisotropic setting we consider also appears in the context of magnetic resonance imaging (MRI) of the human brain (cf. \cite{MM, HS}), anomalous diffusion (cf. \cite{OTSBG}), biological tissues (cf. \cite{OTSBG, HM}), financial mathematics (see \cite{R, CLU}).

The main difference between the fractional Laplacian $(- \Delta)^{s}$ and the anisotropic fractional Laplacian $(- \Delta)^{\beta, s}$ is the geometry determined by the kernel
$$
\mathcal{K}(y) = \dfrac{1}{\Vert y \Vert^{c+s}}.
$$
In the seminal work \cite{CLU}, this anisotropic geometry required a refinement of the techniques presented in \cite{CS}: for example, a new covering lemma and a suitable scaling. Recently, in \cite{LEITAO}, the second author studied an extension problem related to anisotropic fractional Laplacian and a riemannian metric $g$ was crucial to get an anisotropic version of the Almgren's frequency formula obtained in \cite{CS 1}.

The paper is divided into two parts. In the sequel, we comment on the strategies to achieve our results:
\vspace{0,1cm}\\
1. (Smooth solution). In the first part of the paper, we will show that the De Giorgi's approach, see \cite{DEGIORGI, LANDIS}, allows us to reach the $C^{\gamma}$-regularity for smooth solutions $u$ of \eqref{MAIN EQUATION}, where the estimates do not depend on the norm of any derivative or modulus of continuity of $u$. As in \cite{S}, we will control the behavior of a solution $u$ of \eqref{MAIN EQUATION} away from the origin to obtain a Growth Lemma and use an iterate argument to get the desired regularity. In this analysis, two tools are crucial:  barrier function and suitable scaling. In fact, in order to find an appropriate way to control the behavior of $u$ away from the origin in the isotropic case \cite{S}, Silvestre established an interesting inequality involving radial barriers $\eta$ and the kernel $\mathcal{K}$: 
\vspace{0,3cm}\\
\textbf{Silvestre inequality}. Given a $\delta> 0$, there exist $\kappa > 0$ and $\tau >0 $ only depending on $\beta$, dimension $n$, $s$ and $\delta$ such that for all $r>0$ and $x_{} \in \mathbb{R}^{n}$:
\begin{eqnarray}
\label{Fundamental estimate}
\kappa L_{r}\eta(x) + 2\int_{\mathbb{R}^{n} \setminus B_{\frac{1}{4}}}  (\vert 8 y \vert^{\tau} - 1 )\mathcal{K}_{}(ry)r^{n}dy <  \dfrac{1}{2} \inf_{\mathcal{B} \subset B_{2}, \ \vert \mathcal{B}  \vert < \delta}\int_{\mathcal{B}}\mathcal{K}_{}(ry)r^{n}dy, 
\end{eqnarray}
where 
$$L_{r}v(x) :=  \int_{\mathbb{R}^{n}}\left( v\left( x + y \right) - v\left(x \right) - \chi_{B_{1}}\left( ry \right) \nabla v \left( x \right)\cdot y \right) \mathcal{K}_{}\left( ry\right)r^{n}dy.$$
The Silvestre inequality reveals the appropriate scaling for our analysis: the scaling determined by the kernel $\mathcal{K}$. Furthermore, the barrier functions $\eta$ should satisfy the bounds:
\begin{eqnarray}
\label{}
-C \leq L_{r}\eta(x) \leq C,
\end{eqnarray}
for some positive constant $C$ depending on $\beta$, dimension $n$, and $s$. In our case, we will use radial functions as barrier functions and the anisotropic scaling $T_{\beta, r}: \mathbb{R}^{n} \rightarrow \mathbb{R}^{n}$ defined by
\begin{eqnarray}
\label{}
T_{\beta, r} e_{i} = r^{\frac{2}{b_{i}}} e_{i},
\end{eqnarray}
where $e_{i}$ is the $i$-th canonical vector, to get the anisotropic Silvestre inequality and access to the $C^{\gamma}$-regularity.

2. (Viscosity solution). In the second part of the paper, we get the regularity theory established in \cite{CS, CLU} for viscosity solutions of non-local Isaac's equation like the anisotropic fractional Laplacian

\begin{equation}\label{def F1}
\mathcal{I}u\left( x \right) := \inf\limits_{\alpha}\sup\limits_{\beta} L_{\alpha \beta}u\left( x \right)= 0, 
\end{equation}
where $L_{\alpha, \beta}$ is as in \eqref{MAIN EQUATION}.  An important example of the equation \eqref{def F1} was studied in \cite{CLU}. In fact, if
\begin{eqnarray} 
\label{CLU} 
b_{i} = n + \sigma_{i} \quad \text{and} \quad s = 2 - c
\end{eqnarray}
where $\sigma_{i} \in (0, 2)$ we have
\begin{eqnarray}
c_{\sigma} = q_{\max, s} \quad \text{and} \quad \Vert y \Vert^{c+s} =  \sum_{i=1}^{n} \vert y_{i} \vert^{n+\sigma_{i}} 
\end{eqnarray}
for $\sigma=\left( \sigma_{1}, \dots, \sigma_{n} \right)$. In \cite{CS, CLU}, the key that gives access to the regularity theory to viscosity solutions $u$ of the equation \eqref{def F1} is a non-local ABP estimate. In \cite{CLU}, the correct geometry to reach a non-local ABP estimate for integro-differential equation governed by anisotropic kernels $\mathcal{K}_{\alpha, \beta}$ was discovered. More precisely, the geometry determined by the level sets of the kernels $\mathcal{K}_{\alpha \beta}$:
$$
 \Theta_{r}\left( x \right) := \left\lbrace \left( y_{1}, \dots, y_{n}\right) \in \mathbb{R}^{n} : \Vert y - x \Vert  < r^{} \right\rbrace.
$$

With this geometry at hand, three steps are fundamental to obtain a non-local ABP estimate, a Harnack Inequality and the desired regularity:

\begin{enumerate}

\item $u$ \textbf{stays quadratically close} to the tangent plane to concave envelope $\Gamma$ of $u$ in a (large) portion of the neighbourhoods of the contact points and such that, in smaller neighbourhoods (with the same geometry), the concave envelope $\Gamma$ has quadratic growth: here, our neighbourhoods are ellipses $E_{r, 1}$ with the same geometry of $\Theta_{r}$.

\item \textbf{Covering Lemma}. Since our neighbourhoods will be ellipses $E_{r,1}$, our covering is naturally made of $n$-dimensional rectangles $\mathcal{R}_{r}$ and we invoke a covering lemma from \cite{CCal}.

\item \textbf{A barrier function}. We use the natural anisotropic scaling $T_{\beta, r}$ and a radial function to build an adequate barrier function and, together with the nonlocal anisotropic version of the ABP estimate, we get a lemma that links a pointwise estimate with an estimate in measure, Lemma \ref{Point Estimates 1}. This is the crucial step towards a regularity theory. The iteration of Lemma \ref{Point Estimates 1} implies the decay of the distribution function $\lambda_{u}:= \left| \left\lbrace u > t \right\rbrace \right |$ and the tool that makes this iteration possible is the so called Calder\'on -Zygmund decomposition. Since our scaling is anisotropic we need a Calder\'on -Zygmund decomposition for $n$-dimensional rectangles generated by our scaling. A fundamental device we use for that decomposition is the Lebesgue differentiation theorem for $n$-dimensional rectangles that satisfy the condition of Caffarelli-Calder\'on in \cite{CCal}. Hence we obtain the Harnack inequality and, as a consequence, we achieve the interior $C^{\gamma}$ regularity for a solution $u$ of equation \eqref{def F1} and, under additional assumptions on the kernels $\mathcal{K}_{\alpha \beta}$, interior $C^{1, \gamma}$ estimates.

\end{enumerate}

Finally, we emphasize that the restriction $0 < s < 4 / b_{\max}$ in our results comes from the class of solutions $u$ we are studying: solutions of class $C^{2}$ or viscosity solutions ($u$ is touched by a $C^{2}$ function). However, we believe that the results obtained here can naturally be extended for $0 < s < 2$ if we consider an appropriate class of solutions $u$ and change the metric of $\mathbb{R}^{n}$, a namely, $(\mathbb{R}^{n}, g$), where $g$ is the metric determined by kernel $\Vert \cdot \Vert$, see \cite{LEITAO}. We plan to address this issue in a forthcoming paper. Furthermore, the Lemma \cite{CCal2} allows the homogeneity degrees $b_{i}$ depend on $x$, see \cite{CTU}. We would also like to mention that in \cite{CK} an important regularity theory for integro-differential equations was developed, where the kernels are singular, and only charge the coordinate axes for the jumps, and each axis may charge jumps with a different exponent.

The paper is organized as follows. In section \ref{Preliminaries} we gather all the necessary tools for our analysis: fundamental geometry, Silvestre inequality, the notion of viscosity solution for the problem \eqref{def F1}, the extremal operators of Pucci type associated with the family of kernels $K_{\alpha \beta}$ and some notation. In Section \ref{H Regularity: smooth solutions} we present the proof of $C^{\gamma}$-regularity of smooth solutions and as a corollary we get a result type Liouville. The Section \ref{H Regularity: viscosity solutions} is divided in three subsections: \ref{ABP Estimate}, where the nonlocal ABP estimate for a solution $u$ of equation \eqref{def F1} is obtained, is the most important of the paper. Sections \ref{Barrier function section} and \ref{Harnack Inequality Section} are devoted to the proof of the Harnack inequality and its consequences. 
\section{Preliminaries}
\label{Preliminaries}

In this section we gather anisotropic versions of some results obtained in \cite{S, CLU}. We begin with geometric informations that we will systematically use along the work. \\

Given $ r, l > 0$ and $x \in \mathbb{R}^{n}$, we will denote 

$$
E^{}_{r, l}\left( x \right) := \left\lbrace \left( y_{1}, \dots, y_{n}\right) \in \mathbb{R}^{n} : \sum_{i=1}^{n} \frac{\left( y_{i} - x_{i} \right)^{2}}{r^{\frac{4}{b_{i}}}} < l^{2} \right\rbrace.
$$
If $b_{\min}=\min\left\lbrace b_{1}, \dots, b_{n} \right\rbrace$ and $b_{\max} = \max \left\lbrace b_{1}, \dots, b_{n} \right\rbrace$ we define
$$
R^{}_{r,l}\left( x \right) := \left\lbrace \left( y_{1}, \dots, y_{n}\right) \in \mathbb{R}^{n} : | y_{i} - x_{i} | < l^{\frac{2}{b_{\min}}} r^{\frac{2}{b_{i}}} \right\rbrace 
$$
and 
$$
E^{\max}_{r, l}\left( x \right) := \left\lbrace \left( y_{1}, \dots, y_{n}\right) \in \mathbb{R}^{n} : \sum_{i=1}^{n} \frac{\left( y_{i} - x_{i} \right)^{2}}{r^{2\frac{b_{\max}}{b_{i}}}} < l^{2} \right\rbrace.
$$
Furthermore, if $\mathfrak {C}  = \mathfrak {C}> 0$ is a natural number and the $n$-dimensional rectangle
$$
R(x):= \left\lbrace \left( y_{1}, \dots, y_{n}\right) \in \mathbb{R}^{n} : | y_{i} - x_{i} | < l_{i} \right\rbrace 
$$
satisfies
$$
R^{}_{}\left( x \right) \subset \left\lbrace \left( y_{1}, \dots, y_{n}\right) \in \mathbb{R}^{n} : | y_{i} - x_{i} | < 2^{-\mathfrak{C}(k + 1)} r^{\frac{2}{b_{i}}} \right\rbrace, 
$$
for some number natural $k$, we define the corresponding $n$-dimensional rectangle  $\tilde{R}^{}_{}(x)$ by
$$
\tilde{R}^{}_{}\left( x \right) := \left\lbrace \left( y_{1}, \dots, y_{n}\right) \in \mathbb{R}^{n} : | y_{i} - x_{i} | <   \left[ 2^{-\mathfrak{C}\left( \frac{b_{\min}}{2}\right)k} r \right]^{\frac{2}{b_{i}}} \right\rbrace. 
$$

We will also consider the notation 
$$B_{r} = B_{r}(0), \ \ \ \Theta_{r} = \Theta_{r}(0)  \ \ \ \text{and} \quad E^{\max}_{r, l} = E^{\max}_{r, l}(0).$$

The geometric properties of the sets defined above will be crucial in our analysis. We collect them in the following Lemma.
\begin{lemma}[Fundamental Geometry]
\label{Fundamental Geometry}
Let $r > 0$ and $l>0$. Then, given $x \in \mathbb{R}^{n}$, we have the following relations:
\begin{enumerate}
\item $E^{}_{r,1}(x) \subset \Theta_{r\sqrt{n}} \subset E_{r\mathfrak{C}, 1}(x)$ and $E_{2^{-\mathfrak{C}}r}(x) \subset E^{}_{r,\frac{1}{4}}(x)$, for some natural number $\mathfrak{C}=\mathfrak{C}\left( n, b_{\max} \right) > 0$.
\item If $R$ is a $n$-dimensional rectangle, then $R^{}_{} (x) \subset \tilde{R}^{}_{}(x)$. Moreover, $R_{r,l}(x) \subset E_{(rl)c_{\max},1}(x)$, where $c_{\max} = n^{\frac{b_{\max}}{4}}$, if $r,l \in (0,1)$.
\item $E^{\max}_{\frac{r}{2}, 1}(x) \subset E^{\max}_{r,1/2}(x)$ and $E^{\max}_{r, l}(x) \subset E^{\max}_{rl,1}(x)$, if $l \geq 1$.
\item If $\tau_{1}$ is the topology generated by Euclidean balls $B_{r}(z)$ and $\tau_{2}$ is the topology generated by anisotropic balls $\Theta_{r}(z)$, then $\tau_{1} = \tau_{2}$.
\item If $T_{\beta, r}: \mathbb{R}^{n} \rightarrow \mathbb{R}^{n}$ is defined by
\begin{eqnarray}
T_{\beta, r}e_{i} = r^{\frac{2}{b_{i}}}e_{i} \quad  \text{or} \quad T^{}_{\max, r^{}}e_{i} = r^{\frac{b_{\max}}{b_{i}}}e_{i},
\end{eqnarray}
where $e_{i}$ is the i-th canonical vector, then $T_{\beta, r}(B_{l}) = E_{r, l}$ or $T^{}_{\max, r^{}}(B_{l}) = E^{\max}_{r, l}$.
\end{enumerate}
\end{lemma}

Next we will divide this section into two subsections: Smooth solutions and Viscosity solutions and extremal operators.

\subsection{Smooth solutions}\label{preliminaries ss}

Without loss of generality, we consider $L = (- \Delta)^{\beta, s}$. In this subsection, we establish the tools to get the regularity $C^{\gamma}$ for $\Delta^{\beta, s}$-harmonic smooth functions. Precisely, we show that the operator  $\Delta^{\beta, s}$ applied to radial functions $\eta$ is bounded for $ s \in \left( 0, 4/b_{\max} \right)$ and we get the Silvestre inequality for $\Delta^{\beta, s} \eta$.

\begin{lemma}[Barrier function]
\label{Barrier function}
Let $\eta: \mathbb{R}^{n}  \rightarrow \mathbb{R}^{}$ defined by
\begin{eqnarray}
\eta(y) = \left \{ 
\begin{array}{ll}
(1 - \vert y \vert^{2})^{2}, \ \text{ if } \ y \in B_{1}, \\
\\
0, \ \text{ if } \ y \in ( \mathbb{R}^{n} \setminus B_{1}).
\end{array}
\right.
\end{eqnarray}
There exist $C > 0$ only depending on $\beta$ , dimension $n$ and $s$ such that 
\begin{eqnarray}
\vert (- \Delta)^{\beta, s} \eta(x) \vert \leq C \quad \text{for all} \ x \in B_{3/4}.
\end{eqnarray}

\end{lemma}
\begin{proof}
Choose $r_{0}=r_{0}(n, \beta) \in (0,1)$ such that 
\begin{eqnarray}
\vert \eta(x) - \eta (x+y) + \langle \nabla \eta (x), y \rangle\vert \leq C \vert y \vert^{2} \quad \text{for all} \ (x, y) \in B_{3/4} \times E_{r_{0}, 1},
\end{eqnarray}
where $C$ is a positive constant only depending on $\beta$, and dimension $n$. Denote $T_{r_{k}}:= T_{\beta, r_{k}}$, where $r_{k} = r_{0}2^{-k}$. Then, we get
\begin{eqnarray}
\int_{E_{r_{0}, 1}} \dfrac{ \vert \eta(x) - \eta (x + y) +  \langle \nabla \eta (x), y \rangle \vert}{ \Vert y \Vert^{c+s}}dy & \leq & C \nonumber \int_{E_{r_{0}, 1}} \dfrac{ \vert y \vert^{2} }{\Vert y \Vert^{c+s}}dy \\
\nonumber 
& = C & \sum_{k=0}^{\infty} \int_{E_{r_{k}, 1} \setminus E_{r_{k+1}, 1}} \dfrac{ \vert y \vert^{2} }{\Vert y \Vert^{c+s}}dy \\ \nonumber
& = C & \sum_{k=0}^{\infty} r_{k}^{-s} \int_{B_{1} \setminus E_{1/2, 1}} \dfrac{ \vert T_{r_{k}}y \vert^{2} }{\Vert y \Vert^{c+s}}dy \\ \nonumber
\end{eqnarray}
and we can estimate
\begin{eqnarray}
\sum_{k=0}^{\infty} r_{k}^{-s} \int_{B_{1} \setminus E_{1/2, 1}} \dfrac{ \vert T_{r_{k}}y \vert^{2} }{\Vert y \Vert^{c+s}}dy & \leq& \nonumber \sum_{i=0}^{\infty} r_{k}^{q_{\max, s}}\int_{B_{1} \setminus E_{1/2, 1}} \dfrac{ \vert y \vert^{2} }{\Vert y \Vert^{c+s}}dy \\ \nonumber
& = & \dfrac{C(n, \beta, s)}{1 - 2^{-q_{\max, s}}} ,
\end{eqnarray}
where $C(n, \beta, s) = \displaystyle\int_{B_{1} \setminus E_{1/2, 1}} \dfrac{ \vert y \vert^{2} }{\Vert y \Vert^{c+s}}dy$. On the other hand, if $r_{1} = r_{1}(r_{0}) > 0$ is such that $\Theta_{r_{1}} \subset E_{r_{0}, 1}$, we obtain
\begin{eqnarray}
\int_{\mathbb{R}^{n} \setminus E_{r_{0}, 1}} \dfrac{ \vert \eta(x) - \eta(x + y) \vert}{ \Vert y \Vert^{c+s}}dy  & \leq&   2\Vert \eta \Vert_{\infty} \int_{\mathbb{R}^{n} \setminus \Theta_{r_{1}}} \dfrac{1}{ \Vert y \Vert^{c+s}}dy \\ \nonumber
& = & r_{1}^{c} \int_{\mathbb{R}^{n} \setminus \Theta_{1}} \dfrac{1}{ \Vert y \Vert^{c+s}}dy \\ \nonumber
& = & \dfrac{C(n, \beta)}{s}. 
\end{eqnarray}
Then, we find
\begin{eqnarray}
\label{EUCLIDEAN BARRIER}
\vert(- \Delta)^{\beta, s}\eta(x) \vert & \leq & C.
\end{eqnarray}
\end{proof}

Taking into account \eqref{EUCLIDEAN BARRIER} we get the Silvestre inequality for $\Delta^{\beta, s}$:

\begin{lemma}[Silvestre inequality] Given a $\delta> 0$, there exist $0 < \kappa < \frac{1}{4}$ and $\tau >0 $ only depending on $\beta$, dimension $n$, $s$ and $\delta$ such that
\begin{eqnarray}
\label{Fundamental estimate}
\kappa (- \Delta)^{\beta, s}\eta(x) + 2\int_{\mathbb{R}^{n} \setminus B_{\frac{1}{4}}}  (\vert 8 y \vert^{\tau} - 1 )\mathcal{K}_{0}(y)dy < \dfrac{1}{2} \inf_{\mathcal{B} \subset B_{2}, \ \vert \mathcal{B}  \vert < \delta}\int_{\mathcal{B}}\mathcal{K}_{0}(y)dy, 
\end{eqnarray}
for all $x \in B_{3/4}$, where $\mathcal{K}_{0}(y) := \frac{1}{\Vert y \Vert^{c+s}}$ for all $y \in \mathbb{R}^{n} \setminus \left\lbrace 0 \right\rbrace$.
\end{lemma}

\subsection{Viscosity solutions and extremal operators}\label{preliminaries vs}

In this subsection we collect the technical properties of the operator $\mathcal{I}$ that we will use throughout the paper. Since $\mathcal{K}_{\alpha\beta}$ is symmetric and positive, we obtain
$$
L_{\alpha \beta}u\left( x \right) =  PV \int_{\mathbb{R}^{n}}\left( u\left( x + y \right) - u\left(x \right) \right) \mathcal{K}_{\alpha\beta}\left( y\right)dy
$$
and
$$
L_{\alpha \beta}u\left( x \right) =  \dfrac{1}{2} \int_{\mathbb{R}^{n}}\left( u\left( x + y \right) - u\left(x - y \right) - 2\left( x \right) \right) \mathcal{K}_{\alpha\beta}\left( y\right)dy.
$$
For convenience of notation, we denote
$$
\delta \left( u, x, y \right):= u \left( x + y \right) + u \left( x - y \right) - 2u \left( x \right)
$$
and we can write
$$
L_{\alpha \beta} u\left( x \right) = \int_{\mathbb{R}^{n}}\delta \left( u, x, y \right)\mathcal{K}_{\alpha\beta}\left( y\right)dy,
$$
for some kernel $\mathcal{K}_{\alpha\beta}$.

We now define the adequate class of test functions for our operators.
 
\begin{definition}
\label{C^{1,1} definition}
A function $\varphi$ is said to be $C^{1,1}$ at the point $x$, and we write $\varphi\in C^{1,1}\left( x \right)$, if there is a vector $v \in \mathbb{R}^{n}$ and numbers $M, \eta_{0} >0$ such that 
$$
\vert \varphi \left( x + y\right) - \varphi \left( x \right) -  v \cdot y   \vert \leq M \vert y \vert^{2},
$$
for $\vert  x \vert < \eta_{0}$. We say that a function $\varphi$ is $C^{1,1}$ in a set $\Omega$, and we denote $\varphi \in C^{1,1}\left( \Omega \right)$, if the previous holds at every
point, with a uniform constant $M$.
\end{definition}

\begin{remark}
Let $u \in C^{1,1}\left( x \right) \cap L^{\infty}\left( \mathbb{R}^{n} \right)$ and $M>0$ and $\eta_{0}>0$ be as in definition \ref{C^{1,1} definition}. Then, by Lemma \ref{Barrier function}, we find
$$L_{\alpha \beta}u\left( x \right)  =  PV \int_{\mathbb{R}^{n}}\delta \left( u, x, y \right) \mathcal{K}_{\alpha\beta}\left( y\right)dy \leq C(n, \Lambda, b_{\min}, b_{\max}, \eta_{0}, s).$$
\end{remark}

We now introduce the notion of viscosity subsolution (and supersolution) $u$ in a domain $\Omega$, with $C^{2}$ test functions that touch $u$ from above or from below. We stress that $u$ is allowed to have arbitrary discontinuities outside of $\Omega$.

\begin{definition}
Let $f$ be a bounded and continuous function in $\mathbb{R}^{n}$. A function $u:\mathbb{R}^{n} \rightarrow \mathbb{R}$, upper (lower) semicontinuous in $\overline{\Omega}$, is said to be a subsolution (supersolution) to equation $\mathcal{I}u = f$, and we write $\mathcal{I}u \geq f$ ($\mathcal{I}u \leq f$), if whenever the following happen:  
\begin{enumerate}
\item $x_{0} \in \Omega$ is any point in $\Omega$;
\medskip
\item $B_{r}\left( x_{0} \right) \subset \Omega$, for some $r>0$;
\medskip
\item $\varphi \in C^{2}\left(\overline{B_{r}\left( x_{0} \right)} \right)$;
\medskip
\item $\varphi \left( x_{0} \right) = u\left( x_{0} \right)$;
\medskip
\item $\varphi \left( y \right) > u\left( y \right)$ ($\varphi \left( y \right) < u\left( y \right)$) for every $y \in B_{r}\left( x_{0} \right) \setminus \left\lbrace x_{0} \right\rbrace$; 
\end{enumerate}
then, if we let 
$$
v := \left \{ 
\begin{array}{lll}
\varphi , & \text{ in } & B_{r}\left( x_{0} \right) \\
u & \text{ in } & \mathbb{R}^{n} \setminus  B_{r}\left( x_{0} \right), 
\end{array}
\right.
$$
we have $\mathcal{I}v\left( x_{0} \right) \geq f\left( x_{0} \right)$ ($\mathcal{I}v\left( x_{0} \right) \leq f\left( x_{0} \right)$).
\end{definition}

\begin{remark}
 Functions which are $C^{1,1}$ at a contact point $x$ can be used as test functions in the definition of viscosity solution (see Lemma 4.3 in \cite{CS}).
\end{remark}
Next, we define the class of linear integro-differential operators that will be a fundamental tool for the regularity analysis.

\begin{definition}
Let $\mathfrak{L}_{0}$ be the collection of linear operators $L_{\alpha \beta}$. We define the maximal and minimal operator with respect to $\mathfrak{L}_{0}$ as
$$
\mathcal{M}^{+} u \left( x \right) := \sup\limits_{L \in \mathfrak{L}_{0}} Lu\left( x \right) 
$$
and
$$
\mathcal{M}^{-} u \left( x \right) := \inf \limits_{L \in \mathfrak{L}_{0}} Lu\left( x \right).
$$
\end{definition}
By definition, if $\mathcal{M}^{+} u \left( x \right) < \infty$ and $\mathcal{M}^{-} u \left( x \right) < \infty$, we get 
$$
\mathcal{M}^{+} u \left( x \right) = q_{\max, s} \int_{\mathbb{R}^{n}} \dfrac{\Lambda \delta^{+} - \lambda \delta^{-}}{\Vert y \Vert^{c+s}} dy
$$
and
$$
\mathcal{M}^{-} u \left( x \right) = q_{\max, s} \int_{\mathbb{R}^{n}} \dfrac{\lambda \delta^{+} - \Lambda \delta^{-}}{\Vert y \Vert^{c+s}} dy.
$$

The proofs of the results that we now present can be found in the sections $3$, $4$ and $5$ of \cite{CS}. The first result ensures that if $u$ can be touched from above, at a point $x$, with a paraboloid then $Iu\left( x\right)$ can be evaluated classically.

\begin{lemma}
\label{clas sense max}
If we have a subsolution, $\mathcal{I}u \geq f$ in $\Omega$, and $\varphi$ is a $C^{2}$ function that touches $u$ from above at a point $x \in \Omega$, then $\mathcal{I}u\left( x \right)$ is defined in the classical sense and $\mathcal{I}u\left( x\right) \geq f\left( x \right)$.
\end{lemma}

Another important property of $\mathcal{I}$ is the continuity of $\mathcal{I}\varphi$ in $\Omega$ if $\varphi\in C^{1, 1}\left( \Omega \right)$.

\begin{lemma}
\label{I is C^{1,1}}
Let $v$ be a
bounded function in $\mathbb{R}^{n}$ and $C^{1,1}$ in some open set $\Omega$. Then $\mathcal{I}v$ is continuous in $\Omega$.
\end{lemma}

The next lemma allows us to conclude that the difference between a subsolution of the maximal operator $\mathcal{M}^{+}$ and a supersolution of the minimal operator $\mathcal{M}^{-}$ is a subsolution of the maximal operator.

\begin{lemma}
\label{Comp. princ.}
Let $\Omega$ be a bounded open set and $u$ and $v$ be two bounded functions in $\mathbb{R}^{n}$ such that

\begin{enumerate}

\item $u$ is upper-semicontinuous and $v$ is lower-semicontinuous in $\overline{\Omega}$;
\medskip
\item $\mathcal{I}u \geq f$ and $\mathcal{I}v \leq g$ in the viscosity sense in $\Omega$ for two continuous functions $f$ and $g$.

\end{enumerate}
Then 
$$\mathcal{M}^{+}\left( u - v \right) \geq f - g \quad \mathrm{in} \ \ \Omega$$
in the viscosity sense.

\end{lemma}

\section{H\"older Regularity: smooth solutions}
\label{H Regularity: smooth solutions}

As in \cite{S} we will use the De Giorgi's approach to achieve the $C^{\gamma}$-regularity for $\Delta^{\beta, s}$-harmonic smooth functions. We begin with a \textit{Growth lemma}.
\begin{lemma}[Growth lemma] 
\label{anisotropic-conditions}
If $u$ is a function that satisfies:
\begin{enumerate}
\item $(- \Delta)^{\beta, s}u \leq 0$ in $B_{1}$;
\item $u \leq 1$ in  $B_{1}$;
\item $u(x) \leq 2 \vert  2 x \vert^{\tau} - 1$ for all $x \in \mathbb{R}^{n} \setminus B_{1}$;
\item $\vert \left\lbrace x \in B_{1} :  u(x) \leq 0 \right\rbrace\vert > \delta$.
\end{enumerate}
Then, there exists a constant $\mu = \mu(n, s, \beta, \delta ) > 0$ such that $u \leq 1 - \mu$ in $B_{1/2}$. 
\end{lemma}
\begin{proof}
Consider $\mu =  \kappa ( \eta (1/2) - \eta(3/4) )$. Suppose, for the purpose of contradiction, that there exists $x_{0} \in B_{\frac{1}{2}}$ such that
\begin{eqnarray}
u(x_{0}) > 1 - \mu = 1 -  \kappa \eta \left(1/2\right) +  \kappa \eta(3/4).
\end{eqnarray}
Thus, since $\eta$ is decreasing in any ray from the origin and $u \leq 1$ in  $B_{1}$, we have
\begin{eqnarray}
v(x_{0})  > v(x), \quad \text{for all} \ x \in B_{1} \setminus B_{\frac{3}{4}},
\end{eqnarray}
where $v(x) = u(x) + \kappa \eta(x)$. Then, we conclude that 
\begin{eqnarray}
1 < \sup_{x \in B_{1}}v(x) = v(x_{1})
\end{eqnarray}
for some $x_{1} \in B_{\frac{3}{4}}$. If we define
$$\mathcal{B}_{} = \left\lbrace y \in \mathbb{R}^{n} : x_{1} + y \in B_{1} \right\rbrace \quad \text{and} \quad \mathcal{B}_{0} = \left\lbrace y \in \mathbb{R}^{n} : x_{1} + y \in B_{1}, \ u(x_{1} + y) \leq 0 \right\rbrace$$
we can write
\begin{eqnarray}
(- \Delta)^{\beta, s}v(x_{1}) & = & \int_{\mathbb{R}^{n}} (v\left(x_{1} \right) - v\left( x_{1} + y \right) ) \mathcal{K}_{0}(y)dy \\ \nonumber
& = & I_{1} + I_{2},
\end{eqnarray}
where we denote
\begin{eqnarray*}
I_{1}  =   \int_{\mathcal{B}_{}} (v\left(x_{1} \right) - v\left( x_{1} + y \right) ) \mathcal{K}_{0}(y)dy \quad \text{and} \quad I_{2} = \int_{\mathbb{R}^{n} \setminus \mathcal{B}_{}} (v\left(x_{1} \right) - v\left( x_{1} + y \right) )\mathcal{K}_{0}(y)dy.
\end{eqnarray*}
Since $v$ has a maximum at $x_{1}$ and $v(x_{1}) \geq 1$ we estimate
\begin{eqnarray*}
I_{1} & = &  \int_{\mathcal{B}_{0}} (v\left(x_{1} \right) - v\left( x_{1} + y \right) )\mathcal{K}_{0}(y)dy +  \int_{\mathcal{B}_{} \setminus \mathcal{B}_{0}} (v\left(x_{1} \right) - v\left( x_{1} + y \right) ) \mathcal{K}_{0}(y)dy \\ \nonumber
& \geq &  \int_{\mathcal{B}_{0}} (v\left(x_{1} \right) - v\left( x_{1} + y \right) )\mathcal{K}_{0}(y)dy \\ \nonumber
& \geq & \int_{\mathcal{B}_{0}} (1 - \kappa \eta (x_{1} + y) ) \mathcal{K}_{0}(y)dy \\ \nonumber
& \geq & \dfrac{1}{2} \int_{\mathcal{B}_{0}} \mathcal{K}_{0}(y)dy. 
\end{eqnarray*}
Using the conditions 2 and 3 we find
\begin{eqnarray*}
I_{2} & = & \int_{\mathbb{R}^{n} \setminus \mathcal{B}_{}} (v\left(x_{1} \right) - v\left( x_{1} + y \right) )\mathcal{K}_{0}(y)dy  \\ \nonumber
& \geq & \int_{\mathbb{R}^{n} \setminus \mathcal{B}_{}} \left [1 -  (2 \vert  2 (x_{1} + y )\vert^{\tau} -1 )\right ]\mathcal{K}_{0}(y)dy \\ \nonumber
& = & \int_{\mathbb{R}^{n} \setminus \mathcal{B}_{}} \left [2 -  2^{\tau + 1} \vert  x_{1} + y \vert^{\tau} \right ]\mathcal{K}_{0}(y)dy  \\ \nonumber
& \geq & \int_{\mathbb{R}^{n} \setminus \mathcal{B}_{}} \left [2 -  2^{\tau + 1} ( 3/4 + \vert y \vert)^{\tau} \right ]\mathcal{K}_{0}(y)dy.
\end{eqnarray*}
Moreover, since $(\mathbb{R}^{n} \setminus \mathcal{B}_{} ) \subset (\mathbb{R}^{n} \setminus B_{1/4})$ we obtain

\begin{eqnarray*}
I_{2} & \geq  & \int_{\mathbb{R}^{n} \setminus \mathcal{B}_{}} \left [2 -  2^{\tau + 1} \vert  3/4 + y \vert^{\tau} \right ]\mathcal{K}_{0}(y)dy \\ \nonumber
& = & \int_{\mathbb{R}^{n} \setminus B_{\frac{1}{4}}} \left [2 -  2^{\tau + 1} \vert  3/4 + y \vert^{\tau} \right ]\mathcal{K}_{0}(y)dy - \int_{(\mathbb{R}^{n} \setminus B_{1/4}) \cap \mathcal{B}_{}} \left [2 -  2^{\tau + 1}( 3/4 + \vert y \vert)^{\tau} \right ]\mathcal{K}_{0}(y)dy \\ \nonumber
&\geq & \int_{\mathbb{R}^{n} \setminus B_{\frac{1}{4}}} \left [2 -  2^{\tau + 1} ( 3/4 + \vert y \vert)^{\tau}\right ]\mathcal{K}_{0}(y)dy.
\end{eqnarray*}
From condition 1 we have
$$(- \Delta)^{\beta, s}v (x_{1}) = (- \Delta)^{\beta, s} (u(x_{1}) + \kappa \eta(x_{1})) \leq \kappa (- \Delta)^{\beta, d}\eta(x_{1})$$
and using the condition 4 we obtain
\begin{eqnarray*}
\kappa (- \Delta)^{\beta, s}\eta(x_{1}) \geq 2\int_{\mathbb{R}^{n} \setminus B_{\frac{1}{4}}}  (1 - \vert 8 y \vert^{\tau})K_{0}(y)dy + \dfrac{1}{2} \inf_{\mathcal{B} \subset B_{2}, \ \vert \mathcal{B}  \vert > \delta}\int_{\mathcal{B}}\mathcal{K}_{0}(y)dy, 
\end{eqnarray*}
which contradicts \eqref{Fundamental estimate}.
\end{proof}

Using the anisotropic scaling $T^{}_{\max, r^{}}$ and  Lemma \ref{anisotropic-conditions} we get the following scaled version.
\begin{lemma}[Growth lemma-anisotropic] 
 \label{anisotropic-conditions 2}
If $u$ is a function that satisfies:
\begin{enumerate}
\item $(- \Delta)^{\beta, s}u \leq 0$ in $E^{\max}_{r,1}(x_{0})$;
\item $u \leq C$ in  $E^{\max}_{r,1}(x_{0})$;
\item $u(x) \leq C \left (2 \vert  2T^{-1}_{\max, r^{}}(x - x_{0})\vert^{\tau} - 1 \right)$ for all $x \in \mathbb{R}^{n} \setminus E^{\max}_{r,1}(x_{0})$;
\item $\dfrac{\vert \left\lbrace x \in E^{\max}_{r,1}(x_{0}) :  u(x) \leq 0 \right\rbrace\vert}{r^{\frac{b_{\max}}{2}c}}> \delta$.
\end{enumerate}
Then, there exists a constant $\mu = \mu(n, s, \beta, \delta) > 0$ such that $u \leq C(1 - \mu) $ in $E^{\max}_{\frac{r}{2}, 1}$. 
\end{lemma}
\begin{proof}
Define
\begin{eqnarray}
v(x) = \dfrac{u(T^{}_{\max, r^{}}x + x_{0})}{C},
\end{eqnarray}
for all $x \in \mathbb{R}^{n}$. Since $T^{}_{\max, r^{}}(B_{1}) = E^{\max}_{r,1}$ we conclude that $v$ satisfies 2 and 3. Furthermore, we find
\begin{eqnarray}
(- \Delta)_{}^{\beta, s}v(x) \leq 0 \quad \text{and} \quad \vert \left\lbrace x \in B_{1} :  v(x) \leq 0 \right\rbrace\vert > \delta.
\end{eqnarray}
By Lemma \ref{anisotropic-conditions} there exists a constant $\mu = \mu(n, s, \beta) > 0$ such that $v \leq 1 - \mu$ in $B_{1/2}$. Thus, we find $u \leq C(1 - \mu) $ in $E^{\max}_{r, 1/2}$. Finally, by Lemma \ref{Fundamental Geometry} we have $E^{\max}_{\frac{r}{2}, 1} \subset E^{\max}_{r, 1/2}$ and the Lemma \ref{anisotropic-conditions 2} is concluded.
\end{proof}

\begin{theorem}
\label{Theorem particular case}
If $u$ is a bounded function that satisfies $(- \Delta)^{\beta, s}u = 0$ in $E^{\max}_{2r, 1}$, then for $\delta=\frac{\vert B_{1} \vert}{2}$ there exist constants $\gamma = \gamma (n, s, \beta) \in (0,1)$ and $C=C(n, s, \beta) >0$ such that 
\begin{eqnarray}
\sup_{x, y \in E^{\max}_{r, 1}}\dfrac{\vert u(x) - u(y) \vert}{\Vert x - y \Vert^{\gamma}} \leq \dfrac{C}{r^{\gamma}}\Vert u \Vert_{\infty}.
\end{eqnarray}
In particular, $u \in C_{loc}^{\frac{\gamma b_{\min}}{2}}(E^{\max}_{r,1})$. 
\end{theorem}
\begin{proof}
By considering the anisotropic scaling $v(x) = u(T^{}_{\max, r^{}}x)/2 \Vert u \Vert_{\infty}$ we can suppose that $\text{osc}_{\mathbb{R}^{n}} u = 1$ and $r = 1$. As in \cite{S}, given $x_{0} \in B_{1}$ we will construct a nondecreasing sequence $c_{k}$ and a nonincreasing sequence $d_{k}$ such that $d_{k} - c_{k} = 2^{-k\alpha}$
\begin{eqnarray}
\label{Cond. ind. 1}
d_{k} - c_{k} = 2^{-k\alpha} \quad \text{and} \quad c_{k} \leq u \leq d_{k} \quad \text{in} \ E^{\max}_{r_{k}, 1}(x_{0}),
\end{eqnarray}
where $r_{k} = r_{0}^{k}$ for any integer number $k$ and $0 < \alpha < 1$ will be chosen appropriately. Now we consider two cases:\\
\vspace{0,3cm}\\
Case 1: $k \leq 0$.\\

Since $\text{osc}_{\mathbb{R}^{n}} u = 1,$ we can write 
\begin{eqnarray}
c_{k} = \inf_{\mathbb{R}^{n}} u \quad \text{and} \quad d_{k} = c_{k} + r_{k}^{\alpha}, 
\end{eqnarray}
for $k \leq 0$ and for all $\alpha \in (0, 1)$.\\
\vspace{0,3cm}\\
Case 2: $k \geq 1$.\\

Suppose that we already have $c_{j}$ and $d_{j}$ for $j = 1, \dots, k$. We will find $c_{k+1}$ and $d_{k+1}$ satisfying \eqref{Cond. ind. 1}. In fact, if
\begin{eqnarray}
\mathfrak{m} = \dfrac{c_{k} + d_{k}}{2} 
\end{eqnarray}
then by \eqref{Cond. ind. 1} we find
\begin{eqnarray}
\vert u - \mathfrak{m} \vert \leq \dfrac{2^{-k\alpha}}{2}  \quad \text{in} \ E^{\max}_{r_{k}, 1}(x_{0}).
\end{eqnarray}
Now define
\begin{eqnarray}
v(x) = 2\dfrac{(u(x) - \mathfrak{m})}{r_{k}^{\alpha}},
\end{eqnarray}
for all $x \in E^{\max}_{r_{k}, 1}(x_{0})$. Clearly, we have 
\begin{eqnarray}
\vert v  \vert \leq 1 \quad \text{in} \ E^{\max}_{r_{k}, 1}(x_{0})
\end{eqnarray}
and 
\begin{eqnarray}
(- \Delta)^{\beta, d} v  \leq 0 \quad \text{in} \ E^{\max}_{r_{k}, 1}(x_{0}).
\end{eqnarray}
Next, we will analysis two cases:\\
\vspace{0,3cm}\\
(i) Assume that 
\begin{eqnarray}
\dfrac{\vert \left\lbrace x \in E^{\max}_{r_{k}, 1}(x_{0}) :  v(x) \leq 0 \right\rbrace\vert}{r_{k}^{c\frac{b_{\max}}{2}}} \geq \frac{\vert B_{1} \vert}{2}.
\end{eqnarray}
Taking into account that
\begin{eqnarray}
x \in \mathbb{R}^{n} \setminus E^{\max}_{r_{k}, 1}(x_{0})  = T^{-1}_{\max, r_{k}}(\mathbb{R}^{n} \setminus B_{1}(x_{0}))
\end{eqnarray}
we obtain
\begin{eqnarray}
T^{-1}_{\max, r_{k}}(x - x_{0})  \in \mathbb{R}^{n} \setminus B_{1}.
\end{eqnarray}
Thus, there exists $j \in \mathbb{N}$ such that 
\begin{eqnarray}
2^{j} \leq \vert T^{-1}_{\max, r_{k}}(x - x_{0})   \vert \leq 2^{j+1}
\end{eqnarray}
Hence, we find 
\begin{eqnarray}
T^{-1}_{\max, r_{k}}(x - x_{0})   \in B_{2^{(j+1)}} 
\end{eqnarray}
and from Lemma \ref{Fundamental Geometry} 
\begin{eqnarray}
x - x_{0} \in E^{\max}_{r_{k}, 2^{(j+1)}} \subset E^{\max}_{2^{-k + j +1}, 1} = E^{\max}_{r_{(k - j -1)}, 1}.
\end{eqnarray}
Thus, by inductive hypothesis we estimate
\begin{eqnarray}
v(x) & = & 2 \dfrac{(u(x) - \mathfrak{m})}{r_{k}^{\alpha}} \\ \nonumber
& \leq & 2 \dfrac{(a_{k -j -1} - \mathfrak{m})}{2^{-k\alpha}}
\end{eqnarray}
and since $c_{k}$ is a nondecreasing sequence we obtain
\begin{eqnarray}
v(x) &  \leq  & 2  \dfrac{(a_{k -j -1} - \mathfrak{m})}{r_{k}^{\alpha}}.  \\ \nonumber
& = & 2 \dfrac{(a_{k -j -1} - c_{k -j -1} +  c_{k -j -1} - \mathfrak{m})}{r_{k}^{\alpha}}  \\ \nonumber
& \leq & 2 \dfrac{(a_{k -j -1} - c_{k -j -1} + c_{k} - \mathfrak{m})}{r_{k}^{\alpha}} \\ \nonumber
& \leq & 2 \left(\dfrac{2^{-(k -j -1)\alpha}}{r_{k}^{\alpha}} - \dfrac{1}{2}  \right) \\ \nonumber
& = & 2 (22^{j})^{\alpha} - 1,
\end{eqnarray}
for all $x \in \mathbb{R}^{n} \setminus E^{\max}_{r_{k}, 1}(x_{0})$. If we take $\alpha \in \left( 0, \tau \right]$ we get
\begin{eqnarray}
v(x) \leq \left (2 \vert  2 T^{-1}_{\max, r_{k}}(x - x_{0})  \vert^{\tau} - 1 \right) \quad \text{for all} \ x \in \mathbb{R}^{n} \setminus E^{\max}_{r_{k}, 1}(x_{0}).
\end{eqnarray}
Then, we can apply the Lemma \ref{anisotropic-conditions 2} to obtain $v \leq 1 - \mu$ in $E^{\max}_{r_{k}/2, 1} (x_{0})= E^{\max}_{r_{k+1}, 1}(x_{0})$. We then scale back to $u$ to find
\begin{eqnarray}
u \leq c_{k} + \left( \dfrac{2 - \mu}{2} \right) r_{k}^{\alpha} \quad \text{in} \ E^{\max}_{r_{k+1}, 1}(x_{0}).
\end{eqnarray}
Now we define $c_{k+1} = c_{k}$ and $d_{k} = c_{k} + r_{k+1}^{\alpha}$. Clearly, $c_{k+1} \leq u$ in $E^{\max}_{r_{k+1}, 1}(x_{0})$. Finally, if we choose $\alpha = \min \left\lbrace \tau, \frac{\ln (1 - \mu/2)}{\ln 2} \right\rbrace$ we obtain 
\begin{eqnarray}
u \leq d_{k+1} \quad \text{in} \ E^{\max}_{r_{k+1}, 1}(x_{0}).
\end{eqnarray}
\vspace{0,3cm}\\
(ii) In the case
\begin{eqnarray}
\dfrac{\vert \left\lbrace x \in E^{\max}_{r_{k}, 1}(x_{0}) :  v(x) \leq 0 \right\rbrace\vert}{r_{k}^{c\frac{b_{\max}}{2}}}  <  \frac{\vert B_{1} \vert}{2}
\end{eqnarray}
we consider $v = - u$ to obtain
\begin{eqnarray}
u \geq d_{k} - \left( \dfrac{2 - \mu}{2} \right) r_{k}^{\alpha} \quad \text{in} \ E^{\max}_{r_{k+1}, 1}(x_{0}).
\end{eqnarray}
Now we define $d_{k+1} = d_{k}$ and $c_{k+1} = d_{k} - \left( \dfrac{2 - \mu}{2} \right) r_{k}^{\alpha}$.

Finally, given $x_{0} \in B_{1}$ and $y \in \mathbb{R}^{n}$ we can choose an integer $k$ such that $x_{0} - y \in (E^{\max}_{r_{k-1}, 1} \setminus E^{\max}_{r_{k}, 1})$. Thus, by Lemma \ref{Fundamental Geometry}  we can conclude
\begin{eqnarray}
\vert u(x_{0}) - u (y) \vert \leq r_{k-1}^{\alpha} \leq C \Vert x_{0} - y \Vert^{\gamma} ,
\end{eqnarray}
where $C=C(n, \alpha, b_{\min}, b_{\max}) >1$ and $\gamma = \frac{ 2\alpha}{b_{\max}}$.
\end{proof}


\begin{corollary}[Liouville property]
Let $u$ be a bounded function that satisfies $(- \Delta)^{\beta, s}u = 0$ in $\mathbb{R}^{n}$. Then, $u$ is constant.
\end{corollary}

\begin{proof}
Given $x, y \in \mathbb{R}^{n}$, choose $R > 0$ such that $x, y \in E^{\max}_{R, 1}$. By Theorem \ref{Theorem particular case} we have
\begin{eqnarray}
\dfrac{\vert u(x) - u(y) \vert}{\Vert x - y \Vert^{\gamma}} \leq \dfrac{C}{R^{\gamma}}\Vert u \Vert_{\infty}.
\end{eqnarray}
Taking $R > 0$ large enough, we get $u(x) = u(y)$. Hence, $u$ is constant.
\end{proof}

\section{H\"older Regularity: viscosity solutions}

\label{H Regularity: viscosity solutions}

In this section, we obtain the ingredients necessary to reach the interior $C^{\gamma}$ and $C^{1, \gamma}$ regularity for viscosity solutions of $\mathcal{I}u = 0$.
\subsection{Nonlocal anisotropic ABP estimate} \label{ABP Estimate}

In this subsection we get an ABP estimate for integro-differential equations like anisotropic fractional Laplacian. 
 
Let $u$ be a non positive function outside the ball $B_{1}$. We define the concave envelope of $u$ by 
$$
\Gamma \left( x \right) := \left \{ 
\begin{array}{lll}
\min \left\lbrace p\left( x\right): \ \text{for all planes} \ p \geq u^{+} \ \text{in} \ B_{3} \right\rbrace , & \text{ in } & B_{3} \\
\\
0 & \text{ in } & \mathbb{R}^{n}\setminus B_{3}. 
\end{array}
\right. 
$$  
 
\begin{lemma} \label{cov. 1 lemma} Let $u \leq 0$ in $\mathbb{R}^{n}\setminus B_{1}$ and $\Gamma$ be its concave envelope. Suppose $f \in L^{\infty}$ and $\mathcal{M}^{+}u\left( x \right) \geq -f\left( x\right)$ in $B_{1}$. Let $\rho_{0} = \rho_{0}\left( n \right) > 0$,
$$
 r_{k} := \rho_{0}2^{- \left( \frac{1}{q_{\min,s}}\right)}  2^{- \mathfrak{C}\left( \frac{b_{\min}}{2} \right) k  },
$$
where $\mathfrak{C}= \mathfrak{C}(b_{\min}, b_{\max})$ is a natural number such that 
$$
E_{lr, 1} \subset E_{r, 1/ 2},
$$
with $l = 2^{- \mathfrak{C}\left[ \frac{b_{\min}}{2} \right] }$ for all $r > 0$ and $q_{\min,s} = \frac{4}{b_{\min}} - s$. Given $M>0$, we define
$$W_{k}\left( x \right)  := E_{r_{k}, 1}\setminus E_{r_{k+1}, 1} \cap \left\lbrace y : u\left( x + y\right) < u\left( x \right) + \langle y, \nabla \Gamma \left( x \right)\rangle - M\left( \frac{q_{\min,s}}{q_{\max,s}}\right)r^{\frac{4}{b_{\min}}}_{k} \right\rbrace.$$ 
Then there exists a constant $C_{0}>0$, depending only on $n$, $\lambda$, $b_{\min}$ and  $b_{\max}$, such that, for any $x \in \left\lbrace u = \Gamma \right\rbrace $ and any $M>0$, there is a $k$ such that
\begin{equation}
\label{meas est 1}
\left | W_{k}\left( x \right)  \right | \leq C_{0} \frac{f\left( x \right)}{M} \left | E_{r_{k}, 1}\setminus E_{r_{k+1}, 1} \right |. 
\end{equation}
\end{lemma} 
\begin{proof}
Notice that $u$ is touched by the plane  
$$
\Gamma \left( x \right) + \langle y - x, \nabla \Gamma \left( x \right) \rangle 
$$
from above at $x$. From Lemma \ref{clas sense max}, $\mathcal{M}^{+} u \left( x \right)$ is defined classically and we get
\begin{equation}
\label{max exp}
\mathcal{M}^{+} u \left( x \right) =  q_{\max,s} \int_{\mathbb{R}^{n}} \dfrac{\Lambda \delta^{+} - \lambda \delta^{-}}{\Vert y \Vert^{c + s}} dy.
\end{equation}
We will show that 
\begin{equation}
\label{non posit delta}
\delta\left( y \right):= \delta \left( u, x, y \right) = u \left( x + y \right) + u \left( x - y \right) - 2u \left( x \right) \leq 0.
\end{equation}
In fact, if both $x-y \in B_{3}$ and $x+y \in B_{3}$ then we conclude that $\delta \left( y \right) \leq 0 $, since $u\left( x \right) = \Gamma \left( x \right) = p\left(  x \right)$, for some plane $p$ that remains above $u$ in the whole ball $B_{3}$. Moreover, if either $x-y \notin B_{3}$ or $x+y \notin B_{3}$, then both $x-y$ and $x+y$ are not in $B_{1}$, and thus $u\left( x + y \right) \leq 0 $ and $u\left( x - y \right) \leq 0$. Therefore, in any case the inequality \eqref{non posit delta} is proved. Combining \eqref{max exp} and \eqref{non posit delta}, we find
\begin{eqnarray}
\label{cov. 1 lemma est1}
- f\left( x \right) & \leq & \mathcal{M}^{+} u \left( x \right) \nonumber\\ 
& = & q_{\max,s} \int_{E_{r_{0}, 1}} \dfrac{ - \lambda \delta^{-} }{\Vert y \Vert^{c + s}} dy, 
\end{eqnarray}
where $r_{0} = \rho_{0}2^{- \frac{1}{q_{\min,s}}}$. Since $x \in \left\lbrace u = \Gamma \right\rbrace$, we would like to emphasize that $ y \in W_{k}\left( x \right) $ implies $-y \in W_{k}\left( x \right)$. Hence, we find
\begin{equation}
\label{cov. 1 lemma est3}
W_{k}\left( x \right) \subset E_{r_{k}, 1} \setminus E_{r_{k+1}, 1}\cap \left\lbrace y : - \delta \left( y \right) > 2 M \left(  \frac{q_{\min,s}}{q_{\max,s}} \right) r_{k}^{\frac{4}{b_{\min}}} \right\rbrace.
\end{equation}
Using \eqref{cov. 1 lemma est1}, we estimate
\begin{eqnarray}
\label{cov. 1 lemma est2}
f\left( x \right) & \geq & c\left(n, \lambda\right)\left[  q_{\max,s}  \sum_{k=0}^{\infty}\int_{E_{r_{k}, 1}\setminus E_{r_{k+1}, 1}} \dfrac{  \delta^{-} }{\Vert y \Vert^{c + s}}dy \right]  \nonumber \\ 
& \geq &  c\left(n, \lambda\right) \sum_{k=0}^{\infty}\left[ q_{\max,s} (n^{-\frac{c+s}{2}})r^{-(c + s)}_{k}\int_{W_{k}} \delta^{-} dy \right] . 
\end{eqnarray}
Moreover, we have
\begin{eqnarray*}
 \left |    E_{r_{k}, 1} \setminus E_{r_{k+1}, 1}  \right | = \left( \prod_{j=1}^{n} r_{k}^{\frac{2}{b_{i}}} \right)  \left |    B_{1} \setminus E_{l, 1}  \right | = r^{c}_{k} \left |    B_{1} \setminus E_{l, 1}  \right |,
\end{eqnarray*}
where $l =  2^{- \mathfrak{C}\frac{b_{\min}}{2}}$. Therefore, we find
\begin{eqnarray*}
\label{ABP NEW 1}
 \left |    E_{r_{k}, 1} \setminus E_{r_{k+1}, 1}  \right | \geq c(b_{\min}, b_{max}) r^{c}_{k}.
\end{eqnarray*}
Let us assume by contradiction that \eqref{meas est 1} is not valid. Then, from \eqref{cov. 1 lemma est3}, \eqref{cov. 1 lemma est2} and \eqref{ABP NEW 1}, we obtain
\begin{eqnarray}
f\left( x \right)  & \geq & \nonumber  c_{1}\left(n, \lambda, b_{\min}, b_{\max}\right)  \left[ q_{\min,s}  \sum_{k=0}^{\infty} \left( 2 M r^{q_{\min, s}}_{k} f(x) \frac{C_{0}}{M} \right) \right]  \\ \nonumber
& = & c_{2}(n, \lambda, b_{\min}, b_{\max}) f(x) C_{0} \left[ q_{\min,s}   \sum_{k=0}^{\infty} (2 r^{q_{\min, s}}_{k})\right]  \\ \nonumber
& \geq & c_{3}(n, \lambda, b_{\min}, b_{\max}) f(x) C_{0} \rho_{0}^{ q_{\min,s} }   \left[ q_{\min,s} \sum_{k=0}^{\infty} 2^{-\left( q_{\min,s}  \right)k}\right]  \\ \nonumber
& \geq & c_{2}(n, \lambda, b_{\min}, b_{\max}) f(x) C_{0}  \rho_{0}^{ \frac{4}{b_{\min}} }  \left[ q_{\min,s} \sum_{k=0}^{\infty} 2^{-\left( q_{\min,s}  \right)k}\right]  \\ \nonumber
& \geq &  c_{3}(n, \lambda, b_{\min}, b_{\max})f(x) C_{0}  \left[ q_{\min,s} \sum_{k=0}^{\infty} 2^{-\left( q_{\min,s}  \right)k}\right]. 
\end{eqnarray}
Then, we get
\begin{eqnarray*}
f\left( x \right) & \geq & \dfrac{c_{3} C_{0}q_{\min,s} f\left( x \right)}{1- 2^{-q_{\min,s}}}.
\end{eqnarray*}
Finally, since $\frac{q_{\min,s} }{1- 2^{-q_{\min,s}}}$ is bounded away from zero, for all $s \in \left( 0, \dfrac{4}{b_{\max}} \right)$, we find
$$
f\left( x \right)  \geq  c_{4}\left(n, \lambda, b_{min}, b_{\max} \right)C_{0}f\left( x \right),
$$
which is a contradiction if $C_{0}$ is chosen large enough.
\end{proof}

As in \cite{CLU}, the following result is a direct consequence of the arguments used in the proof of \cite[Lemma 8.4]{CS}.

\begin{lemma} \label{cov. 2 lemma, SC} Let $\Gamma$ be a concave function in $B_{1}$ and $v\in \mathbb{R}^{n}$. Assume that, for a small $\varepsilon > 0$,
$$
\left |\left(  B_{1} \setminus B_{\frac{1}{2}} \right) \cap \left\lbrace y : \Gamma \left( y\right) < \Gamma \left( 0 \right) + \langle T\left( y \right), v \rangle - h \right\rbrace \right | \leq \varepsilon \left | B_{1} \setminus B_{\frac{1}{2}} \right |,
$$
where $T:\mathbb{R}^{n} \rightarrow \mathbb{R}^{n}$ is a linear map. Then 
$$\Gamma \left(  y \right) \geq \Gamma (0) + \langle T\left( y \right), v \rangle - h$$ 
in the whole ball $B_{\frac{1}{2}}$.
\end{lemma}
\begin{proof}
Let $y \in B_{\frac{1}{2}}$. There exist $B_{\frac{1}{2}}\left( y_{1}\right) \subset B_{1} \setminus B_{1/2}$ and $B_{\frac{1}{2}}\left( y_{2}\right) \subset B_{1} \setminus B_{1/2}$ such that 
$$
L\left( B_{\frac{1}{2}}\left( y_{1}\right) \right) = B_{\frac{1}{2}}\left( y_{2}\right),
$$
where $L: B_{\frac{1}{2}}\left( y_{1}\right) \rightarrow  B_{\frac{1}{2}}\left( y_{2}\right)$ is the linear map
$$
L\left( z \right) = 2y - z.
$$
Geometrically, the balls $B_{\frac{1}{2}}\left( y_{1}\right)$ and $B_{\frac{1}{2}}\left( y_{2}\right)$ are symmetrical with respect to $y$. Then, if $\varepsilon > 0$ is sufficiently small, there will be two points $z_{1} \in B_{\frac{1}{2}}\left( y_{1}\right)$ and $z_{2} \in B_{\frac{1}{2}}\left( y_{2}\right)$ such that 
\begin{enumerate}
\item $y = \dfrac{z_{1}+z_{2}}{2}$;
\medskip
\item $\Gamma\left( z_{1} \right) \geq \Gamma\left( 0 \right) + \langle T\left( z_{1} \right), v \rangle - h$;
\medskip
\item $ \Gamma\left( z_{2} \right) \geq \Gamma\left( 0 \right) + \langle T\left( z_{2} \right), v \rangle - h$. 
\medskip
\end{enumerate} 
Hence, since $T$ and $\langle \cdot, v \rangle$ are linear maps and $\Gamma$ is a concave function, we obtain
$$
\Gamma\left( y \right) \geq \Gamma\left( 0 \right) + \langle T\left( y \right), v \rangle - h.
$$
\end{proof}

As in \cite{CLU} , we use Lemma \ref{cov. 2 lemma, SC} to prove the version of Lemma 8.4 in \cite{CS} for our problem.

\begin{lemma}
\label{cov. 2 lemma}
Let $r>0$ and $\Gamma$ be a concave function in $E_{r, 1}$. There exists $\varepsilon_{0} > 0$ such that if
$$
\left |E_{r,1} \setminus E_{r, \frac{1}{2}} \cap \left\lbrace y : \Gamma \left( y\right) < \Gamma \left( 0 \right) + \langle y, \nabla \Gamma \left(0 \right)\rangle - h \right\rbrace \right | \leq \varepsilon \left |  E_{r,1} \setminus E_{r, \frac{1}{2}} \right |,
$$
for $0< \varepsilon \leq \varepsilon_{0}$, then 
$$\Gamma \left(  y \right) \geq \Gamma \left( 0 \right) + \langle y, \nabla \Gamma \left( 0 \right)\rangle - h$$ 
in the whole set $E_{r, \frac{1}{2}}$.
\end{lemma}

\begin{proof}
Consider 
$$
\mathcal{A}:= \left(  B_{1} \setminus B_{\frac{1}{2}}\right) \cap \left\lbrace y : \tilde{\Gamma} \left( y\right) < \tilde{\Gamma} \left( 0 \right) + \langle T_{\beta, r}\left( y \right), \nabla \Gamma \left( 0 \right) \rangle - h \right\rbrace 
$$
and
$$
\mathcal{D}:= E_{r,1} \setminus E_{r, \frac{1}{2}} \cap \left\lbrace y : \Gamma \left( y\right) < \Gamma \left( 0 \right) + \langle y, \nabla \Gamma \left(0 \right)\rangle - h \right\rbrace.
$$
Notice that
$$
 \mathcal{A} = T_{\beta, r}^{-1}\left( \mathcal{D} \right),  
$$
where $\tilde{\Gamma}\left( x \right) := \Gamma \left(  T_{\beta, r} \left( x \right) \right)$. Moreover,
$$
 B_{1} \setminus B_{\frac{1}{2}}  = T_{\beta, r}^{-1}\left( E_{r, 1} \setminus E_{r, \frac{1}{2}} \right) \quad \text{and} \quad   B_{\frac{1}{2}}  = T_{\beta, r}^{-1}\left( E_{r, \frac{1}{2}} \right). 
$$
Then, taking into account that $\tilde{\Gamma}$ is concave, the lemma follows from Lemma \ref{cov. 2 lemma, SC}.
\end{proof}

\begin{corollary}
\label{ABP cor 1}
Let $\varepsilon_{0} > 0$ be as in Lemma \ref{cov. 2 lemma}. Given $0 < \varepsilon \leq \varepsilon_{0}$, there exists a constant $C\left( n, \lambda, b_{\min}, b_{\max}, \varepsilon \right) > 0$ such that for any function $u$ satisfying the same hypothesis as in Lemma \ref{cov. 1 lemma}, there exist $r \in \left( 0, \rho_{0} 2^{-\frac{1}{q_{\min, s}} } \right)$ and $k= k\left( x \right)$ such that 
$$\left | E_{r, 1} \setminus E_{l r, \frac{1}{2}} \cap \left\lbrace y : u\left( x + y\right) < u\left( x \right) + \langle y, \nabla \Gamma \left( x \right)\rangle - C\left( \frac{q_{\min, s}}{q_{\max,s}}\right)  f\left( x \right)\sum \limits_{i=1}^{n} r^{\frac{4}{b_{i}}} \right\rbrace \right |   $$
\begin{equation} \label{ABP cor 1.1}
\leq \varepsilon_{} \left | E_{r, 1} \setminus E_{l r, 1} \right |
\end{equation}
and 
$$
\left | \nabla \Gamma \left( E_{r, \frac{1}{4} }\left( x \right) \right) \right | \leq C\left( \frac{q_{\min, s}}{q_{\max,s}}\right)^{n} f\left( x \right)^{n} \left | E_{r, \frac{1}{4} } \left( x \right) \right |,
$$
where $r = \rho_{0} 2^{-\frac{1}{q_{\min, s}} }2^{- \mathfrak{C}\left[ \frac{b_{\min}}{2} \right] k  }$ and $l =  2^{- \mathfrak{C}\left[ \frac{b_{\min}}{2}\right]}$.
\end{corollary}
\begin{proof}
Taking $M=\frac{C_{0}}{\varepsilon  C^{-1}_{1} }f\left( x \right)$ in Lemma \ref{cov. 1 lemma}, we obtain \eqref{ABP cor 1.1} with $C_{2}:= \frac{C_{0}}{\varepsilon C^{-1}_{1} }$,
where
$$
C_{1} := \dfrac{\left | B_{1} \right | }{\left | B_{1} \setminus B_{1/2} \right | } > 1.
$$
Consider the sets
$$
W_{1, r} :=E_{ r, 1} \setminus E_{r, \frac{1}{2}} \cap \left\lbrace y : \Gamma\left( x + y\right) < u\left( x \right) + \langle y, \nabla \Gamma \left( x \right)\rangle - C_{2}\left( \frac{q_{\min, s}}{q_{\max,s}}\right) f\left( x\right) r^{\frac{4}{b_{\min}}} \right\rbrace   $$
and
$$W_{2, r}\left( x \right) := E_{r, 1} \setminus E_{ lr, 1} \cap \left\lbrace y : u\left( x + y\right) < u\left( x \right) + \langle y, \nabla \Gamma \left( x \right)\rangle - C_{2}\left( \frac{q_{\min, s}}{q_{\max,s}}\right) f\left( x\right) r^{\frac{4}{b_{\min}}} \right\rbrace.$$

Then, since 
$$ E_{ r, 1} \setminus E_{r, \frac{1}{2}}  \subset E_{r, 1} \setminus E_{ lr, 1}, \quad u\left( x \right) = \Gamma\left( x \right), \quad \text{and} \quad u\left( x + y\right) \leq \Gamma \left( x + y\right),$$
for $y \in E_{r, 1}$, we have $W_{1,r} \subset W_{2, r} \subset W_{r}\left( x \right)$. Thus, from  \eqref{ABP cor 1.1} we obtain
\begin{eqnarray}
\left | W_{1,r}\left( x \right) \right |  \leq   \left | W_{2, r}\left( x \right) \right |  \leq  \frac{\varepsilon}{C_{1}} \left | E_{r, 1} \setminus E_{l r, 1} \right |.
\end{eqnarray}
Moreover, we estimate
\begin{eqnarray}
\frac{\varepsilon}{C_{1}} \left | E_{r, 1} \setminus E_{l r, 1} \right | & = & \frac{\varepsilon}{C_{1}} r^{c} \dfrac{\left | B_{1} \setminus E_{l, 1} \right | }{\left | B_{1} \setminus B_{1/2} \right | } \left | B_{1} \setminus B_{\frac{1}{2}} \right | \\  \nonumber
& \leq &  \frac{\varepsilon}{C_{1}} r^{c} C_{1} \left | B_{1} \setminus B_{1/2} \right | \\ \nonumber
& \leq & \varepsilon_{0} \left | E_{r, 1} \setminus E_{r, 1/2} \right |.
\end{eqnarray}

Then, from Lemma \ref{cov. 2 lemma} and the concavity of $\Gamma$, we find
$$
 0 \leq F\left( y \right) \leq  2C_{2}\left( \frac{q_{\min, s}}{q_{\max,s}}\right)f\left( x \right)r^{\frac{4}{b_{\min}}} \quad \ \text{in} \ E_{r, \frac{1}{2}}, 
$$
where 
$$F\left( y \right):= \Gamma \left( x + y \right) - \Gamma \left( x \right) - \langle y,  \nabla \Gamma \left( x \right) \rangle + C_{2}\left( \frac{q_{\min, s}}{q_{\max,s}}\right)f\left( x \right)r^{\frac{4}{b_{\min}}}.$$ 
Notice that
$$
\nabla F\left( x + y \right) = \nabla \Gamma\left( x + y \right) - \nabla \Gamma \left( x \right) .
$$
Then, since $F$ is concave, we find
\begin{eqnarray*}
\left | \nabla \Gamma\left( x + y \right) - \nabla \Gamma \left( x \right) \right | & \leq & \dfrac{\Vert F \Vert_{L^{\infty}\left(E_{r, \frac{1}{2}} \right)}}{\dist \left( \partial E_{r, \frac{1}{2}} , E_{ r, \frac{1}{4}}\right)} \\ 
& \leq & \dfrac{C_{2} f\left( x\right) \left( \frac{q_{\min, s}}{q_{\max,s}}\right)r^{\frac{4}{b_{\min}}}}{\dist \left( \partial E_{r, \frac{1}{2}} , E_{ r, \frac{1}{4}}\right)}  \\ 
& \leq & C_{3}\left( \frac{q_{\min, s}}{q_{\max,s}}\right)f\left( x \right)r^{\frac{2}{b_{\min}}}.
\end{eqnarray*}
Thus, we have
$$
\nabla \Gamma \left( E_{ r, \frac{1}{4}} \right) \subset B_{C_{3}\left( \frac{q_{\min, s}}{q_{\max,s}}\right)f\left( x \right)r^{\frac{2}{b_{\min}}}}\left( \nabla \Gamma \left( x \right) \right)
$$
and obtain
$$
 \left | \nabla \Gamma \left( E_{r, \frac{1}{4}} \right) \right | \leq C_{4}\left( \frac{q_{\min, s}}{q_{\max,s}}\right)^{n}f\left( x \right)^{n} \left | E_{r, \frac{1}{4}}  \right |.    
$$
Finally, taking $C=\max \left\lbrace C_{2}, C_{4} \right\rbrace $, the lemma is proven.
\end{proof}
 
The following covering lemma is a fundamental tool in our analysis.

\begin{lemma}[Covering Lemma, {\cite[Lemma 3]{CCal}}]
\label{covering lemma}
Let $S$ be a bounded subset of $\mathbb{R}^{n}$ such that for each $x \in S$ there exists an $n$-dimensional rectangle $\mathcal{R}\left( x \right)$, centered at $x$, such that:
\begin{itemize}
\item the edges of $\mathcal{R}\left( x \right)$ are parallel to the coordinate axes;
\medskip
\item the length of the edge of $\mathcal{R}\left( x \right)$ corresponding to the $i$-th axis is given by $h_{i}\left(t \right)$, where $t=t\left( x \right)$, $h_{i}\left( t \right)$ is an increasing function of the parameter $t\geq 0$, continuous at $t=0$, and $h_{i}\left( 0 \right)=0$. 
\end{itemize}
Then there exist points $\left\lbrace x_{k} \right\rbrace$ in $S$ such that 
\begin{enumerate}
\item $S \subset \bigcup_{k=1}^{\infty} \mathcal{R}\left( x_{k}\right)$;
\medskip
\item each $x \in S$ belongs to at most $C=C\left( n \right) >0$ different rectangles. 
\end{enumerate}
\end{lemma}

The Corollary \ref{ABP cor 1} and the Covering Lemma \ref{covering lemma} allow us to obtain a lower bound on the volume of the union of the level sets $E_{r, 1}$ where $\Gamma$ and $u$ detach quadratically from the corresponding tangent planes to $\Gamma$ by the volume of the image of the gradient map, as in the standard ABP estimate.
\begin{corollary}
\label{ABP cor 2}
For each $x \in \Sigma = \left\lbrace  u = \Gamma \right\rbrace \cap B_{1}$, let $E_{r, 1}\left( x \right)$ be the level set obtained in Corollary \ref{ABP cor 1}. Then, we have
$$
C\left( \sup \limits u \right)^{n} \leq \left | \bigcup \limits_{x \in \Sigma} E_{r, 1} \left( x \right)\right |.
$$
\end{corollary}

The nonlocal anisotropic version of the ABP estimate now reads as follows.

\begin{theorem}
\label{ABP Nonlocal theorem}
Let $u$ and $\Gamma$ be as in Lemma \ref{cov. 1 lemma}. There is a finite family of open rectangles $\left\lbrace \mathcal{R}_{j}\right\rbrace_{j \in \left\lbrace 1, \dots , m \right\rbrace }$ with diameters $d_{j}$ such that the following hold:
\begin{enumerate}

\item Any two rectangles $\mathcal{R}_{i}$ and $\mathcal{R}_{j}$ in the family do not intersect.

\medskip

\item $\left\lbrace  u = \Gamma \right\rbrace \subset \bigcup_{j=1}^{m} \overline{\mathcal{R}}_{j} $.

\medskip

\item $\left\lbrace  u = \Gamma \right\rbrace \cap \overline{\mathcal{R}}_{j} \neq \emptyset $ for any $\mathcal{R}_{j}$.

\medskip

\item $d_{j} \leq  \sqrt{\sum \limits_{i=1}^{n}\left( \rho_{0} 2^{-\frac{1}{q_{\min, s}}}\right)^{\frac{4}{b_{i}}} }$.

\medskip

\item $ \left | \nabla \Gamma \left( \overline{\mathcal{R}}_{j} \right) \right | \leq C \left( \max_{\tilde{\mathcal{R}}_{j}} f^{+} \right)^{n} \left | \tilde{\mathcal{R}}_{j} \right |$.

\medskip

\item $ \left | \left\lbrace y \in C\tilde{\mathcal{R}}_{j} : u\left( y \right) \geq \Gamma \left( y \right) - C\left( \max_{\tilde{\mathcal{R}}_{j}} f \right) \left( \tilde{d}_{j}\right)^{2} \right\rbrace \right | \geq \varsigma \left | \tilde{\mathcal{R}}_{j} \right |$,

\end{enumerate} 
where $\tilde{d}_{j}$ is the diameter of the rectangle $\tilde{\mathcal{R}}_{j}$ corresponding to $\mathcal{R}_{j}$. The constants $\varsigma > 0$ and $C >0$ depend only on $n$, $\lambda$, $\Lambda$, $b_{\min}$, $b_{\max}$, and $s$.
\end{theorem}
\begin{proof}
We cover the ball $B_{1}$ with a tiling of rectangles of edges 
$$\dfrac{\left( \rho_{0} 2^{-\frac{1}{q_{\min, s}}}\right)^{\frac{2}{b_{i}}}}{2^{\mathfrak{C}}}.$$
We discard all those that do not intersect $\left\lbrace  u = \Gamma \right\rbrace$. Whenever a rectangle does not satisfy (5) and (6), we split its edges by $2^{n\mathfrak{C}}$ and discard those whose closure does not intersect $\left\lbrace  u = \Gamma \right\rbrace$. Now we prove that all remaining rectangles satisfy (5) and (6) and that this process stops after a finite number of steps.

As in \cite{CLU} we will argue by contradiction. Suppose the process is infinite. Then, there is a sequence of nested rectangles $\mathcal{R}_{j}$ such that the intersection of their closures will be a point $x_{0}$. Moreover, since 
$$
\left\lbrace  u = \Gamma \right\rbrace \cap \overline{\mathcal{R}}_{j} \neq \emptyset 
$$ 
and $\left\lbrace  u = \Gamma \right\rbrace$ is closed, we have $x_{0} \in \left\lbrace  u = \Gamma \right\rbrace$. Let $0 < \varepsilon_{1} < \varepsilon_{0}$, where $\varepsilon_{0}$ is as in Corollary \ref{ABP cor 1}. Thus, there exist 
$$r \in \left( 0, \rho_{0} 2^{-\frac{1}{q_{\min, s}} } \right)$$ 
and $k_{0}= k_{0}\left( x_{0} \right)$ such that 
$$\left | E_{r} \setminus E_{rl, 1} \cap \left\lbrace y : u\left( x + y\right) < u\left( x \right) + \langle y, \nabla \Gamma \left( x \right)\rangle - C f\left( x \right)\sum \limits_{i=1}^{n} r^{\frac{4}{b_{i}}} \right\rbrace \right | 
$$
\begin{equation}\label{meas est diam}
\leq \varepsilon_{1}  \left | E_{r, 1} \setminus E_{lr, 1} \right |
\end{equation}
and 
\begin{eqnarray}
\label{meas est diam 2} 
\left | \nabla \Gamma \left( E_{r, 1/4} \left( x_{0} \right) \right) \right | \leq C f\left( x_{0} \right)^{n} \left |  E_{r, 1/4}\left( x_{0} \right)\right |,
\end{eqnarray}
where 
$$r = \rho_{0} 2^{-\frac{1}{q_{\min, s}} }2^{- \mathfrak{C}\left( \frac{b_{\min}}{2} \right) k_{0}}.$$
Let $\mathcal{R}_{j}$ be the largest rectangle in the family containing $x_{0}$ such that
$$
2^{-\mathfrak{C}\left( k_{0} + 2\right)} \left( \rho_{0} 2^{-\frac{1}{q_{\min, s}}} \right)^{\frac{2}{b_{i}}} \leq l_{j} < 2^{-\mathfrak{C}\left( k_{0} + 1\right)} \left( \rho_{0} 2^{-\frac{1}{q_{\min, s}}} \right) ^{\frac{2}{b_{i}}}. 
$$
Thus, from Lemma \ref{Fundamental Geometry} we obtain
$$
\mathcal{R}_{j} \subset E_{r, 1/4} \quad \text{and} \quad E_{r,1} \subset C\tilde{\mathcal{R}}_{j},
$$ 
for some $C=C(n, b_{\min}, b_{\max}) > 1 $. Furthermore, since $\Gamma$ is concave in $B_{2}$, we find
$$
\Gamma \left( y \right) \leq u \left( x_{0} \right) + \langle y - x_{0}, \nabla \Gamma \left( x_{0} \right)\rangle 
$$
in $B_{2}$. Thus, denoting
$$
A_{j}:= \left\lbrace y \in \tilde{\mathcal{R}}_{j} : u\left( y \right) \geq \Gamma \left( y \right) - C\left( \max_{\tilde{\mathcal{R}}_{j}} f \right) \left( \tilde{d}_{j}\right)^{2}  \right\rbrace ,
$$
using \eqref{meas est diam}, \eqref{meas est diam 2}, we obtain
\begin{eqnarray*}
\left | A_{j} \right | & \geq & \left | \left\lbrace y  \in  \tilde{\mathcal{R}}_{j} : u\left( y\right) \geq u\left( x_{0} \right) + \langle y - x_{0}, \nabla \Gamma \left( x_{0} \right)\rangle \right. \right. \\
& & \left. \left. - C f\left( x_{0} \right)\sum \limits_{i=1}^{n} r^{\frac{4}{b_{i}}} \right\rbrace \right | \\ 
& \geq & \left( 1 - \varepsilon_{1} \right) \left | E_{r, 1} \setminus E_{l r, \frac{1}{2}} \right | \\ 
& = & \left( 1 - \varepsilon_{1} \right) r^{c} \left | B_{1} \setminus E_{l, 1} \right | \\ 
& = & \varsigma \left |\tilde{\mathcal{R}}_{j}\right |
\end{eqnarray*}
and
\begin{eqnarray*}
\left | \nabla \Gamma \left( \mathcal{R}_{j} \right) \right | & \leq & \left | \nabla \Gamma \left( E_{r, 1/4} \left( x_{0} \right) \right) \right | \\  
& \leq & C f\left( x_{0} \right)^{n} \left |   E_{r, 1/4} \left( x_{0} \right) \right | \\ 
& = & C f\left( x_{0} \right)^{n} r^{c} \left |   B_{1/4} \left( x_{0} \right) \right | \\
& = & C_{2} f\left( x_{0} \right)^{n} \left | \tilde{\mathcal{R}}_{j} \right |.
\end{eqnarray*} 
Then $\mathcal{R}_{j}$ would not be split and the process must stop, which is a contradiction.
\end{proof}

\begin{remark}
We emphasize that if $b_{\max} = b_{\min} = 2$ we recover the ABP estimate obtained in \cite{CS}. Furthemore, for $b_{\max} = n+ \sigma_{\max}$ and $b_{\min} = n + \sigma_{\min}$ with $\sigma_{\max}, \sigma_{\min} \in (0, 2)$ we have the ABP estimate reached in \cite{CLU}.
\end{remark}

\subsection{A barrier function}\label{Barrier function section}

In order to locate the contact set of a solution $u$ of the maximal equation, as in Lemma \ref{cov. 1 lemma}, we build a barrier function which is a supersolution of the minimal equation outside a small ellipse and is positive outside a large ellipse. 

\begin{lemma}
\label{Barr function 1}
Given $R > 1, $ there exist $p >0$ and $s_{0} \in \left( 0, \frac{4}{b_{\max}}\right)$ such that the function
$$f\left( x \right) = \min \left( 2^{p}, \  | x |^{-p}\right) $$
satisfies
$$\mathcal{M}^{-}f\left( x \right) \geq 0, $$
for $s_{0} < s$ and $1 \leq | x | \leq R $, where $p = p\left( n, \lambda, \Lambda, b_{\min}, b_{\max}, R \right)$, $s_{0}= s_{0}\left( n, \lambda, \Lambda, b_{\min}, b_{\max}, R \right)$.
\end{lemma}
    
\begin{proof}
Consider the following elementary inequalities:
\begin{equation}
\label{bar func: elem ineq 1}
\left( a_{2} + a_{1}\right)^{-l} + \left( a_{2} - a_{1}\right)^{-l} \geq  2a^{-l}_{2} + l\left( l + 1\right)a^{2}_{1}a_{2}^{-l-2} 
\end{equation}
and
\begin{equation}
\label{bar func: elem ineq 2}
\left( a_{2} + a_{1}\right)^{-l} \geq a^{-l}_{2}\left( 1 - l\frac{a_{1}}{ a_{2}}\right). 
\end{equation}
where $0 < a_{1} < a_{2}$ and $l > 0$. Suppose without loss of generality that $b_{1} = b_{\max}$. Taking into account the inequalities \eqref{bar func: elem ineq 1} and \eqref{bar func: elem ineq 2}, we estimate, for $| y | < \frac{1}{2}$, 
\begin{eqnarray*}
\delta (f,e_1,y)& := & \nonumber | e_{1} + y|^{-p} + | e_{1} - y|^{-p} - 2 \\ 
& = & \left(  1 + |y|^{2} + 2y_{1} \right)^{-\frac{p}{2}} +  \left(  1 + |y|^{2} - 2y_{1} \right)^{-\frac{p}{2}} - 2 \\ 
& \geq & 2\left( 1 + |y|^{2} \right)^{-\frac{p}{2}} + p\left( p + 2 \right)y_{1}^{2} \left( 1 + |y|^{2} \right)^{-\frac{p+4}{2}} - 2 \\ 
& \geq &  2\left( 1 - \frac{p}{2}|y|^{2}\right)  +  p\left( p + 2 \right)^{}y_{1}^{2} - p\left( p + 4\right) \frac{\left( p + 2 \right)}{2}y^{2}_{1}|y|^{2}  - 2\\
& = & p \left[  - |y|^{2} + \left( p + 2 \right)y_{1}^{2} - \left( p + 4\right) \frac{\left( p + 2 \right)}{2}y^{2}_{1}|y|^{2}\right].
\end{eqnarray*}
If $1 \leq |x| \leq R $, there is a rotation $T_{x}:\mathbb{R}^{n} \rightarrow \mathbb{R}^{n}$ such that $x = |x| Te_{1}$. Then, changing variables, we obtain
$$
M^{-}f\left( x \right) = q_{\max, s}  |x|^{n-p} \left | \det T_{x} \right | \left[ \int_{\mathbb{R}^{n}} \dfrac{\lambda \delta^{+}\left(f, e_{1}, y \right) - \Lambda\delta^{-}\left(f, e_{1}, y \right)}{\left( \sum_{i=1}^{n}|\left( |x|T_{x}y\right) _{i}|^{b_{i}}\right)^{\frac{c+s}{2}} } dy \right]. 
$$
Thus, we can estimate
\begin{eqnarray}\label{below estimate for min sol}
|x|^{p-n} M^{-}f\left( x \right) & = & \nonumber q_{\max, s} \int_{B_{1/4}\left( 0 \right)}  \dfrac{\Lambda \delta^{+}\left(f, e_{1}, y \right) - \lambda \delta^{-}\left(f, e_{1}, y \right) }{\left( \sum_{i=1}^{n}||x|\left( T_{x}y\right) _{i}|^{b_{i}}\right)^{\frac{c+s}{2}}} dy \nonumber \\ 
& & + \ q_{\max, s} \int_{\mathbb{R}^{n} \setminus B_{1/4}\left( 0 \right)} \dfrac{\Lambda \delta^{+}\left(f, e_{1}, y \right) - \lambda \delta^{-}\left(f, e_{1}, y \right) }{\left( \sum_{i=1}^{n}||x|\left( T_{x}y\right) _{i}|^{b_{i}}\right)^{\frac{c+s}{2}}} dy \nonumber \\ 
& \geq &  q_{\max, s} \int_{B_{1/4}\left( 0 \right)}\dfrac{2p \lambda\left( p + 2 \right)y^{2}_{1}}{\left( \sum_{i=1}^{n}||x|\left( T_{x}y\right) _{i}|^{b_{i}}\right)^{\frac{c+s}{2}}} dy  \nonumber\\
& & -q_{\max, s} \int_{B_{1/4}\left( 0 \right)}\dfrac{2p \Lambda |y|^{2}}{\left( \sum_{i=1}^{n}||x|\left( T_{x}y\right) _{i}|^{b_{i}}\right)^{\frac{c+s}{2}}} dy \nonumber\\ 
&  & - q_{\max, s}\int_{B_{1/4}\left( 0 \right)}\dfrac{\Lambda \frac{1}{2}p\left( p + 4\right) \left( p + 2 \right)|y|^{4}}{\left( \sum_{i=1}^{n}||x|\left( T_{x}y\right) _{i}|^{b_{i}}\right)^{\frac{c+s}{2}}} dy \nonumber\\ 
&  &+ q_{\max, s} \int_{\mathbb{R}^{n} \setminus B_{1/4}\left( 0 \right)} \dfrac{- \lambda 2^{p+1} }{\left( \sum_{i=1}^{n}||x|\left( T_{x}y\right) _{i}|^{b_{i}}\right)^{\frac{c+s}{2}}} dy \nonumber\\ 
& := & I_{1} + I_{2} + I_{3} + I_{4},
\end{eqnarray}
where $I_{1}$, $I_{2}$, $I_{3}$ and $I_{4}$ represent the three terms on the right-hand side of the above inequality. 

Changing variables again, we get

\begin{eqnarray}
\label{NEW BARRIER LS 1}
\int_{B_{1/4}\left( 0 \right)}\dfrac{y^{2}_{1}}{\left( \sum_{i=1}^{n}||x|\left( T_{x}y\right) _{i}|^{b_{i}}\right)^{\frac{c+s}{2}}} dy & = & \int_{T_{x}^{-1}(B_{1/4}\left( 0 \right))}\dfrac{y^{2}_{1}}{\left( \sum_{i=1}^{n}||x|\left( T_{x}y\right) _{i}|^{b_{i}}\right)^{\frac{c+s}{2}}} dy  \\ \nonumber
& = & \int_{B_{1/4}\left( 0 \right)}\dfrac{ \langle \vert x \vert^{-1} T_{x}^{-1}y, e_{1} \rangle^{2}  }{\Vert y \Vert^{c+s}} \vert x \vert^{-n} dy \\ 
& = & \vert x \vert^{-n} \int_{B_{1/4}\left( 0 \right)}\dfrac{ \langle T_{x}^{-1}y,  \vert x \vert^{-1} e_{1} \rangle^{2}  }{\Vert y \Vert^{c+s}}dy \\ 
& = & \vert x \vert^{-n} \int_{B_{1/4}\left( 0 \right)}\dfrac{ \langle y,  x \rangle^{2}  }{\Vert y \Vert^{c+s}}dy . 
\end{eqnarray}
Moreover, without loss of generality, we can assume that 
$$ x \in \left\lbrace y \in \mathbb{R}^{n}: x_{i} \geq 0 \right\rbrace \quad \text{and} \quad x_{1} \geq \frac{1}{n}.$$ From Lemma \ref{Fundamental Geometry} there exists $r_{0} = r_{0}(n, b_{min}, b_{max}) \in (0, 1)$ such that $E_{r_{0}, 1} \subset B_{1/4}$. 
Then, from \eqref{NEW BARRIER LS 1} we estimate
\begin{eqnarray*}
\label{I_{1} estimate}
p^{-1} I_{1} & \geq & q_{\max, s} n^{-1} \lambda \left( p+2\right) |x|^{- \left( n+2\right)} \int_{E_{r_{0}, 1}}\dfrac{ y^{2}_{1}}{\Vert y \Vert^{c+s}} dy.  \\ 
& \geq & c(n,  b_{\min}, b_{\max} )R^{- \left( n+2\right)}n^{-1} \left[ \lambda \left( p+2\right)\int_{\partial B_{1}} y^{2}_{1} d\nu \left( y \right)\right] \left[ \dfrac{r^{q_{\max, s}}_{0}q_{\max, s}}{1 - 2^{-q_{\max, s}}}\right] \\ 
& \geq & C_{3}\left[\left( p+2\right)\int_{\partial B_{1}} y^{2}_{1} d\nu \left( y \right)\right], \\ 
\end{eqnarray*}
where $C_{3}=C_{3}\left( n, \lambda, \Lambda, b_{\min}, b_{\max}, R \right)>0$. Let $C=C\left( n, b_{\max}, b_{\min} \right)$ be a positive constant such that $B_{1/4}\left( 0 \right)  \subset E_{C, 1}$. Then, for $|x| \geq 1$ we get
\begin{eqnarray*}
\label{I_{2} estimate}
p^{-1}I_{2} & \geq & - C_{4} q_{\max, s}  \int_{B_{1/4}\left( 0 \right)}\dfrac{|y|^{2}}{\left( \sum_{i=1}^{n}|\left( T_{x}y\right) _{i}|^{b_{i}}\right)^{\frac{c+s}{2}}} dy \\ 
& = & - C_{4} q_{\max, s} \left | \det T^{-1}_{x} \right |\int_{B_{1/4}\left( 0 \right)}\dfrac{|T_{x}^{-1}y|^{2}}{\Vert y \Vert^{c+s}} dy \\ 
& = & - C_{4} R^{-n} q_{\max, s} \int_{B_{1/4}\left( 0 \right)}\dfrac{|y|^{2}}{\Vert y \Vert^{c+s}} dy \\ 
& \geq & - C_{5} q_{\max, s} \int_{E_{C, 1}}\dfrac{|y|^{2}}{\Vert y \Vert^{c+s}} dy, 
\end{eqnarray*}
where $C_{4}=C_{4}\left( n, \lambda, \Lambda, b_{\min}, b_{\max}, R \right)$.
We have also
$$
q_{\max, s} \int_{E_{C,1}}\dfrac{|y|^{2}}{\Vert y \Vert^{c+s}} dy = q_{\max, s} \sum_{k=1}^{\infty} \int_{E_{r_{k}, 1} \setminus E_{r_{k+1}, 1}}\dfrac{|y|^{2}}{\Vert y \Vert^{c+s}} dy \leq C_{5},
$$
where $r_{k}:= C 2^{-k}$ and $C_{5}=C_{5}\left( n, \lambda, \Lambda, b_{\max}, b_{min} \right)$.
Moreover, we have
\begin{eqnarray}
\label{I_{3} estimate}
I_{3} & \geq & - C_{6}q_{\max.s} \int_{E_{C, 1}}\dfrac{|y|^{4}}{\Vert y \Vert^{c+s}} dy \nonumber \\ 
& \geq & - C_{7}\left[ \dfrac{q_{\max, s}}{1 - 2^{- \left( \dfrac{16}{b_{\max}} - s\right) }}\right]  \\
\end{eqnarray}
and, if $r_{1} = r_{1}(r_{0}) > 0$ is such that $\Theta_{r_{1}} \subset E_{r_{0}, 1}$, we obtain 
\begin{eqnarray}
\label{I_{4} estimate}
I_{4} & \geq & -  C_{8} q_{\max, s} \int_{\mathbb{R}^{n} \setminus E_{r_{0}, 1}}\dfrac{|y|^{4}}{\Vert y \Vert^{c+s}} dy \\ \nonumber
& \geq & - C_{8} q_{\max, s} \int_{\mathbb{R}^{n} \setminus \Theta_{r_{1}}}\dfrac{|y|^{4}}{\Vert y \Vert^{c+s}} dy \\ \nonumber
& \geq &- C_{9}\dfrac{q_{\max, s}}{\left( \frac{16}{b_{\max}}-s\right) } ,
\end{eqnarray}
for positive constants $C_{7}=C_{7}\left( n, \lambda, \Lambda, b_{\min}, b_{\max}, p \right)$ and $C_{8}=C_{8}\left( n, \lambda, \Lambda, b_{\min}, b_{\max}, p \right)$. Choosing $p=p\left( n, \lambda, \Lambda, b_{\min}, b_{\max}, R \right) > 0$ such that
$$
C_{3} \left( p+2\right) \int_{\partial B_{1}} y^{2}_{1} d\nu \left( y \right) - C_{4}C_{5}  > 0 
$$
and combining \eqref{below estimate for min sol}, \eqref{I_{3} estimate} and \eqref{I_{4} estimate}, there is a positive constant $s_{0}=s_{0}\left( n, \lambda, \Lambda, b_{\min}, b_{\max}, R \right) < \frac{4}{b_{\max}}$ such that
$$
|x|^{p-n} M^{-}f\left( x \right)  \geq   C_{9} >0,
$$
for a positive constant $C_{9}=C_{9}\left( n, \lambda, \Lambda,  b_{\min}, b_{\max}, R \right)$.
\end{proof}
 
 As in \cite{CLU}, from Lemma \ref{Barr function 1} we get the following results:    
\begin{corollary}
\label{Coroll Barr function 1}
Given $s_{0} \in \left( 0,  \frac{4}{b_{\max}}\right)$, and $R > 1$, there exist $\kappa>0$ and $p >0$ such that the function
$$
f\left( x \right) = \min \left( \kappa^{-p}, \  | x |^{-p}\right)  
$$
satisfies
$$
\mathcal{M}^{-}f\left( x \right) \geq 0,
$$
for $1  \leq | x | \leq R $ and $s_{0} < s < \frac{4}{b_{\max}}$, where $p = p\left( n, \lambda, \Lambda, b_{\min}, b_{\max}, R \right)$ and $\kappa = \kappa \left( n, \lambda, \Lambda, b_{\min}, b_{\max}, s_{0}, R \right)$.
\end{corollary}

\begin{corollary}
\label{Coroll Barr function 2}
Given $r > 0 $, $R > 1$ and $s_{0} \in \left( 0,  \frac{4}{b_{\max}}\right)$, there exist $\kappa>0$ and $p >0$ such that the function
$$
g\left( x \right) = \min \left( \kappa^{-p}, \  | T_{\beta, r}^{-1}x |^{-p}\right)  
$$
satisfies
$$
\mathcal{M}^{-}g\left( x \right) \geq 0 
$$
for $s_{0} < s < \frac{4}{b_{max}}$ and $ x \in E_{r,R}\setminus E_{r, 1} $, where $p = p\left( n, \lambda, \Lambda, b_{\min}, b_{\max}, R \right)$ and $\kappa = \kappa \left( n, \lambda, \Lambda, b_{\min}, b_{\max}, s_{0}, R \right)$.
\end{corollary}

\begin{lemma}
\label{Exist barr func 5}
Given $s_{0}\in \left(0, \frac{4}{b_{\max}} \right)$, there is a function $\Psi:\mathbb{R}^{n} \rightarrow \mathbb{R}$ satisfying
\begin{enumerate}
\item $\Psi$ is continuous in $\mathbb{R}^{n}$;
\medskip
\item $\Psi = 0$ for $x \in \mathbb{R}^{n} \setminus E_{\frac{1}{4},3\sqrt{n}}$;
\medskip
\item $\Psi> 3$ for $x \in \mathcal{R}_{\frac{1}{4},3}$; 
\medskip
\item $\mathcal{M}^{-}\Psi\left( x\right) > - \varphi \left( x \right) $ for some positive function $\varphi \in C_{0}\left( E_{\frac{1}{4},1} \right)$ for $s_{0} <  s < \frac{4}{b_{\max}}$.
\end{enumerate}
\end{lemma}
\begin{proof}
Consider the function $\Psi:\mathbb{R}^{n} \rightarrow \mathbb{R}$ defined by
$$
\Psi\left( x \right) = \tilde{c} \left \{ 
\begin{array}{lll}
0, & \text{ in } & \mathbb{R}^{n} \setminus E_{\frac{1}{4},3\sqrt{n}} \\
| T^{-1}_{\beta, \frac{1}{4}}x |^{-p} -  \left( 3\sqrt{n}\right)^{-p}  & \text{ in } & E_{\frac{1}{4},3\sqrt{n}} \setminus E_{\frac{1}{4},1} \\
q_{p,\beta},  & \text{ in } & E_{\frac{1}{4},1},
\end{array}
\right.
$$
where $q_{p,\beta}$ is a quadratic function with different coefficients in different directions so that $\Psi$ is $C^{1,1}$ across $E_{\frac{1}{4},1}$.  Choose $ \tilde{c}>0$ such that $\Psi > 3 $ in $\mathcal{R}_{\frac{1}{4},3}$. By Lemma \ref{I is C^{1,1}}, we get
$$\mathcal{M}^{-}\Psi \in C\left( E_{\frac{1}{4},3\sqrt{n}} \right)$$ 
and, from Corollary \ref{Coroll Barr function 2}, we find $\mathcal{M}^{-}\Psi \geq 0$ in $\mathbb{R}^{n} \setminus E_{\frac{1}{4},1}$. The lemma is proved.
\end{proof}
    
\subsection{Harnack inequality and regularity} \label{Harnack Inequality Section}

The next lemma is the fundamental tool towards the proof of the Harnack inequality. It bridges the gap between a pointwise estimate and an estimate in measure.

\begin{lemma} 
\label{Point Estimates 1} 
Let $0 < s_{0} < \frac{4}{b_{\max}}$. If $s \in \left( s_{0}, \frac{4}{b_{\max}} \right) $, then there exist constants $\varepsilon_{0} > 0$, $0 < \varsigma < 1 $, and $M > 1$, depending only $s_{0}, s$, $\lambda$, $\Lambda$, $b_{\min}$, $b_{\max}$, and $n$, such that if
\begin{enumerate}
\item $u \geq 0$ in $\mathbb{R}^{n}$;
\medskip
\item $u\left( 0 \right) \leq 1$;
\medskip
\item $\mathcal{M}^{-}u \leq \varepsilon_{0}$ in $E_{ \frac{3}{4} n^{\frac{b_{\max}}{4}} ,1}$,
\end{enumerate} 
then 
$$| \left\lbrace u \leq M \right\rbrace \cap Q_{1}| > \varsigma.$$
\end{lemma}

\begin{proof}

Let $v= \Psi - u$ and let $\Gamma$ be the concave envelope of $v$ in $E_{\frac{3}{4} n^{\frac{b_{\max}}{4}},3}$. We have 
$$
\mathcal{M}^{+}v \geq  \mathcal{M}^{-}\Psi - \mathcal{M}^{-}u \geq - \varphi - \varepsilon_{0} \quad \text{in} \ \ E_{ \frac{3}{4} n^{\frac{b_{\max}}{4}} ,1}. 
$$
Applying Theorem \ref{ABP Nonlocal theorem} to $v$ (anisotropically scaled), we obtain a family of rectangles $\mathcal{R}_{j}$ such that 
$$
 \sup \limits_{E_{\frac{3}{4} n^{\frac{b_{\max}}{4}},1}}v \leq C \left |\nabla \Gamma \left(E_{\frac{3}{4} n^{\frac{b_{\max}}{4}},1}\right) \right |^{\frac{1}{n}}. 
$$
Thus, by Theorem \ref{ABP Nonlocal theorem} and condition (3) in Lemma \ref{Exist barr func 5}, we obtain
\begin{eqnarray*}
\sup \limits_{E_{\frac{3}{4} n^{\frac{b_{\max}}{4}},1}}v &\leq & C \left |\nabla \Gamma \left(  E_{\frac{3}{4} n^{\frac{b_{\max}}{4}},1} \right) \right |^{\frac{1}{n}} \\ 
&\leq & C_{1} \left( \sum_{i=1}^{m} \left | \nabla \Gamma \left( \mathcal{R}_{j} \right) \right | \right)^{\frac{1}{n}} \\ 
&\leq &  C_{1} \left( \sum_{i=1}^{m} \left( \max_{\tilde{\mathcal{R}}_{j}} \left( \varphi + \varepsilon_{0}\right) ^{+} \right)^{n}\left | \tilde{\mathcal{R}}_{j} \right |\right)^{\frac{1}{n}}\\ 
&\leq &  C_{1}\varepsilon_{0} + \left( \sum_{i=1}^{m} \left( \max_{\tilde{\mathcal{R}}_{j}} \left( \varphi \right) ^{+} \right)^{n}\left | \tilde{\mathcal{R}}_{j} \right |\right)^{\frac{1}{n}}.
\end{eqnarray*}
Furthermore, since $\Psi > 3$ in $E_{\frac{3}{4} n^{\frac{b_{\max}}{4}},1} \supset \mathcal{R}_{\frac{1}{4},3}$ and $u\left( 0 \right) \leq 1$, we get
$$
2 \leq  C_{1}\varepsilon_{0} + \left( \sum_{i=1}^{n} \left( \max_{\tilde{\mathcal{R}}_{j}} \left( \varphi \right)^{+} \right)^{n}\left |\tilde{\mathcal{R}}_{j}\right |\right)^{\frac{1}{n}}.
$$
If $\varepsilon_{0} >0$ is small enough, we have
\begin{equation}
\label{PE 1}
c \leq  \left( \sum \limits_{\tilde{\mathcal{R}}_{j} \cap E_{\frac{1}{4},1}\neq \emptyset}  \left | \tilde{\mathcal{R}}_{j} \right |\right),
\end{equation}
where we used that $\varphi$ is supported in $ E_{\frac{1}{4},1}$. We also have that the diameter of $\tilde{\mathcal{R}}_{j}$ is bounded by $\left( \rho_{0}=\frac{1}{C}\right)^{\frac{2}{b_{\max}}}$. Then, if $\tilde{\mathcal{R}}_{j} \cap E_{\frac{1}{4},1}\neq \emptyset$ we have $C\tilde{\mathcal{R}}_{j} \subset B_{\frac{1}{2}}$. By Theorem \ref{ABP Nonlocal theorem}, we get 
\begin{eqnarray}
\label{PE 2}
& & \left | \left\lbrace y \in C\tilde{\mathcal{R}}_{j}: v\left( y \right) \geq \Gamma \left( y \right) - C\rho_{0}^{\frac{4}{b_{\max}}} \right\rbrace \right | \nonumber \\
& \geq &\left | \left\lbrace y \in C\tilde{\mathcal{R}}_{j} : v\left( y \right) \geq \Gamma \left( y \right) - Cd^{2}_{j}  \right\rbrace \right | \nonumber\\ 
& \geq & \varsigma \left |\tilde{\mathcal{R}}_{j} \right |,
\end{eqnarray}
where we used that $Cd^{2}_{j} < C\rho_{0}^{\frac{4}{b_{\max}}}$. For each rectangles $\tilde{\mathcal{R}}_{j}$ that intersects $E_{\frac{1}{4},1}$ we consider $C\tilde{\mathcal{R}_{j}}$. The family $\left\lbrace C\tilde{\mathcal{R}_{j}}\right\rbrace$ is an  open covering for $\bigcup_{i=1}^{m}\tilde{\mathcal{R}}_{j}$. We consider a subcover with finite overlapping (Lemma \ref{covering lemma}) that also covers $\bigcup_{i=1}^{m}\tilde{\mathcal{R}}_{j}$. Then, using \eqref{PE 1} and \eqref{PE 2} we obtain
\begin{eqnarray*}
\label{PE 3}
& & \left | \left\lbrace y \in B_{\frac{1}{2}} : v\left( y \right) \geq \Gamma \left( y \right) - C\rho_{0}^{\frac{4}{b_{\max}}} \right\rbrace \right | \\
& \geq & \left | \bigcup_{j=1}^{m} \left\lbrace y \in  C\tilde{\mathcal{R}}_{j} : v\left( y \right) \geq \Gamma \left( y \right) - C\rho_{0}^{\frac{4}{b_{\max}}} \right\rbrace \right | \\ 
& \geq & C_{1}\sum_{j=1}^{m}\left | \left\lbrace y \in  C\tilde{\mathcal{R}}_{j} : v\left( y \right) \geq \Gamma \left( y \right) - C\rho_{0}^{\frac{4}{b_{\max}}} \right\rbrace \right |\\ 
& \geq & C_{1}c_{1}.
\end{eqnarray*}
We recall that $B_{\frac{1}{2}} \subset Q_{1}$ and $\Gamma \geq 0$. Hence, if $M:= \sup \limits_{B_{\frac{1}{2}}}\Psi + C\rho_{0}^{\frac{4}{b_{\max}}}$, we have
\begin{eqnarray*}
\label{PE 4}
\left | \left\lbrace y \in Q_{1} : u\left( y \right) \leq M \right\rbrace \right | & \geq & \left | \left\lbrace y \in B_{\frac{1}{2}} : u\left( y \right) \leq M  \right\rbrace \right | \\ 
& \geq &\left | \left\lbrace y \in B_{\frac{1}{2}} : v\left( y \right) \geq \Gamma \left( y \right) - C\rho_{0}^{\frac{4}{b_{\max}}} \right\rbrace \right | \\ 
& \geq & c.
\end{eqnarray*}
\end{proof}

The next lemma is fundamental to iterate Lemma \ref{Point Estimates 1} and to get the $L_{\varepsilon}$ decay of the distribution function $\lambda_{u} := \left |\left\lbrace  u > t \right\rbrace \cap B_{1} \right |$. Since our scaling is anisotropic, the following Calder\'on-Zygmund decomposition is performed with $n$-dimensional rectangles that satisfy the covering lemma of Caffarelli-Calder\'on (Lemma \ref{covering lemma}). We can then apply Lebesgue's differentiation theorem having these $n$-dimensional rectangles as a differentiation basis, see Lemma 5.2 in \cite{CLU}.

\begin{lemma} 
\label{Point Estimates 2}
Let $u$ be as in Lemma \ref{Point Estimates 1}. Then
$$
\left | \left\lbrace u > M^{k} \right\rbrace \cap Q_{1} \right | \leq C\left( 1 - \varsigma \right)^{k}, \quad k=1, \dots,
$$
where $M$ and $\varsigma$ are as in Lemma \ref{Point Estimates 1}. Thus, there exist positive constants $d$ and $\varepsilon$, depending only $s_{0}, s$, $\lambda$, $\Lambda$, $b_{\min}$, $b_{\max}$, and $n$  such that
$$
| \left\lbrace u \geq t \right\rbrace \cap Q_{1}| \leq dt^{-\varepsilon}, \quad \forall t > 0.
$$
\end{lemma}

Using standard covering arguments we get the following theorem.
\begin{theorem}
\label{Point Estimates 3}
Let $u \geq 0$ in $\mathbb{R}^{n}$, $u\left( 0 \right) \leq 1$ and $\mathcal{M}^{-}u \leq \varepsilon_{0}$ in $B_{2}$. Suppose that $ s_{0} < s < \frac{4}{b_{\max}}$ for some $s_{0} >0$. Then
$$
| \left\lbrace u \geq t \right\rbrace \cap B_{1}| \leq Ct^{-\varepsilon}, \quad \forall t > 0,
$$
where $C=C\left( n, \lambda, \Lambda, b_{\max}, b_{\min}, s_{0}, s  \right) >0$ and $\varepsilon=\varepsilon\left( n, \lambda, \Lambda, b_{\max}, b_{\min}, s_{0}, s \right) >0$.
\end{theorem}

\begin{remark}
For each $l > 0$, we will denote $E^{j}_{r,l}:= E_{r^{b_{j}},l}$. Let $u \geq 0$ in $\mathbb{R}^{n}$ and $\mathcal{M}^{-}u \leq C_{0}$ in $E^{j}_{r,2}$, with $0< r \leq 1$. We consider the anisotropic scaling 
$$
v\left( x \right) = \dfrac{u\left( T_{j,\beta, r}x\right) }{u\left( 0 \right) + C_{0}r^{s \frac{b_{j}}{2}} }, \quad x \in \mathbb{R}^{n}, 
$$
where $T_{j,\beta, r}:\mathbb{R}^{n}\rightarrow \mathbb{R}^{n}$ is defined by 
$$
T_{j,\beta, r}e_{i}:= \left \{ 
\begin{array}{lll}
r^{} e_{j}  , & \text{ for } & i=j \\
r^{\frac{b_{j}}{b_{i}}}e_{i}, & \text{ for } & i \neq j. 
\end{array}
\right.
$$
We find $v \geq 0$ in $\mathbb{R}^{n}$, $v\left( 0 \right) \leq 1$ and $T_{j, \beta, r}\left( B_{2}\right) = E^{j}_{r,2}$. Moreover, changing variables, we estimate 
$$\mathcal{M}^{-}v\left(  x \right)  =   \dfrac{r^{s \frac{b_{j}}{2}}}{u\left( 0 \right) + C_{0}r^{s \frac{b_{j}}{2}} } \mathcal{M}^{-}u\left( T_{j, \beta, r}x \right)  \leq  1,$$
for all $x \in B_{2}$.
\end{remark}
Then, using the anisotropic scaling $T_{j, \beta,  r}$ and Theorem \ref{Point Estimates 3} we have the following scaled version.
\begin{theorem} [Pointwise Estimate]
\label{Point Estimates}
Let $u \geq 0$ in $\mathbb{R}^{n}$ and $\mathcal{M}^{-}u \leq C_{0}$ in $E^{j}_{r,2}$. Suppose that $ s_{0} < s < \frac{4}{b_{\max}}$ for some $s_{0} >0$. Then
$$
| \left\lbrace u \geq t \right\rbrace \cap E^{j}_{r,1}| \leq \nonumber C|E^{j}_{r,1}|\left( u\left( 0 \right) + C_{0}r^{s \frac{b_{j}}{2}} \right)^{\varepsilon} t^{-\varepsilon}  \quad \forall t > 0,
$$
where $C=C\left( n, \lambda, \Lambda, b_{\min}, b_{\max}, s_{0}, s \right) >0$ and $\varepsilon=\varepsilon\left( n, \lambda, \Lambda,  b_{\min}, b_{\max}, s_{0}, s \right) >0$.
\end{theorem}

We are now ready to prove the Harnack inequality.

\begin{theorem}[Harnack Inequality]
\label{Harnack Inequality}
Let $u \geq 0$ in $\mathbb{R}^{n}$, $\mathcal{M}^{-}u \leq C_{0}$, and $\mathcal{M}^{+}u \geq -C_{0}$ in $B_{2}$. Suppose that $ s_{0} < s < \frac{4}{b_{\max}}$, for some $s_{0} >0$. Then 
$$u \leq C \left( u\left( 0 \right) + C_{0}  \right) \quad  in \ \ B_{\frac{1}{2}}.$$
\end{theorem}
\begin{proof}
Without loss of generality, we can suppose that $u\left( 0 \right) \leq 1$ and $C_{0} = 1$. Let 
$$\tau = \frac{c b_{\max}}{2 \varepsilon}, $$ 
where $\varepsilon > 0$ is as in Theorem \ref{Point Estimates 3}. For each $\vartheta >0$, we define the function
$$
f_{\vartheta}\left( x \right) := \vartheta \left( 1 - | x | \right)^{-\tau}, \quad x \in B_{1}. 
$$
Let $t>0$ be such that $u \leq f_{t}$ in $B_{1}$. There is an $x_{0} \in B_{1}$ such that $u\left( x_{0} \right) = f_{t}\left( x_{0} \right)$. Let $d := \left( 1 - | x_{0} | \right)$ be the distance from $x_{0}$ to $\partial B_{1}$. 
  
We will estimate the portion of the ellipsoid $E^{\max}_{r, 1}\left( x_{0}\right)$ covered by $\left\lbrace u > \frac{u\left( x_{0} \right) }{2}\right\rbrace $ and by $\left\lbrace u < \frac{u\left( x_{0} \right) }{2}\right\rbrace $. As in \cite{CS}, we will prove that $t>0$ cannot be too large. Thus, since $\tau \leq \dfrac{C(n, b_{\min}, b_{\max})}{\varepsilon}$, we conclude the proof of the theorem. By Theorem \ref{Point Estimates 3}, we have
$$
\left | \left\lbrace u > \frac{u\left( x_{0} \right) }{2}\right\rbrace \cap B_{1} \right | \leq C  \left | \dfrac{2}{u\left( x_{0} \right) }\right |^{\varepsilon} = C t^{-\varepsilon}d^{n} \leq C_{1}  t^{-\varepsilon} r^{\frac{c b_{\max}}{2}}, 
$$
where $r=\frac{d}{2}$. Thus, we get  
\begin{equation}
\label{Harnack Inequality Estimate 1}
\left | \left\lbrace u > \frac{u\left( x_{0} \right) }{2}\right\rbrace \cap E^{\max}_{r, 1}\left(  x_{0} \right)  \right | \leq  C_{1}  t^{-\varepsilon}|E^{\max}_{r, 1}|. 
\end{equation} 
Now we will estimate $\left | \left\lbrace u > \frac{u\left( x_{0} \right) }{2}\right\rbrace \cap E^{\max}_{\theta r, 1}\left( x_{0} \right)  \right |$, where $0 < \theta < 1$. Since 
$$
| x | \leq  | x - x_{0} | + | x_{0} |,  \quad \forall x \in \mathbb{R}^{n},
$$
we have
$$\left( 1 - | x | \right) \geq \left[ d -  \frac{d\theta}{2} \right],$$
for $x \in B_{r \theta }\left( x_{0} \right)$. Hence, if $x \in B_{r \theta }\left( x_{0} \right)$, we get
$$u\left( x \right) \leq \nonumber f_{t}\left( x\right) \leq t \left( 1 - | x | \right)^{-\tau} \leq u\left( x_{0} \right) \left(1 - \frac{\theta}{2} \right)^{-\tau}. $$
Then, since $\mathcal{M}^{+}u \geq - 1$, the function
$$v\left( x \right) =  \left( 1 - \frac{\theta}{2} \right)^{-\tau}u\left( x_{0} \right) - u\left( x \right) $$ 
satisfies
$$v \geq 0 \quad \text{in} \ B_{r \theta }\left( x_{0} \right)  \quad \text{and} \quad \mathcal{M}^{-}v \leq 1.$$
We will consider the function $w:=v^{+}$. For $x \in \mathbb{R}^{n}$ we have
$$\mathcal{M}^{-}w\left( x \right) = \mathcal{M}^{-}v\left( x \right) + \left( \mathcal{M}^{-}w\left( x \right) - \mathcal{M}^{-}v\left( x \right) \right)$$
and
\begin{eqnarray*}
\dfrac{ \mathcal{M}^{-}w\left( x \right) - \mathcal{M}^{-}v\left( x \right)}{q_{\max, s}} & = &  \lambda \int_{\mathbb{R}^{n}} \dfrac{\delta^{+}\left( w, x, y\right) - \delta^{+}\left( v, x, y\right) }{\Vert y \Vert^{c+s}}dy \\ 
&  & + \Lambda \int_{\mathbb{R}^{n}} \dfrac{\delta^{-}\left( v, x, y\right) - \delta^{-}\left( w, x, y\right) }{\Vert y \Vert^{c+s}}dy \\ 
& = & I_{1} + I_{2},
\end{eqnarray*}
where $I_{1}$ and $I_{2}$ represent the two terms in the right-hand side above. Using the elementary equality
$$v^{+}\left( x + y\right) = v\left( x+y\right) + v^{-}\left( x +y \right),$$
and denoting $\delta_{w}:=\delta\left( w, x, y\right)$ and $\delta_{v}:=\delta\left( v, x, y\right)$, we obtain
$$\delta^{+}_{w} = \delta_{v} + v^{-}\left( x - y \right) + v^{-}\left( x + y \right).$$
Thus, taking in account that 
$$\delta_{w}^{+} \geq \delta_{v}^{+} \quad \text{and} \quad \delta_{v} = \delta_{v}^{+} - \delta_{v}^{-},$$
we estimate
\begin{eqnarray}
\label{Harnack Inequality Estimate 6}
I_{1} & = & - \lambda \int_{\left\lbrace \delta_{w}^{+} > \delta_{v}^{+}\right\rbrace } \dfrac{\delta_{v}^{-}}{\Vert y \Vert^{c+s}}dy  \nonumber\\
& & +  \lambda \int_{\left\lbrace \delta_{w}^{+} > \delta_{v}^{+} \right\rbrace } \dfrac{v^{-}\left( x + y\right) + v^{-}\left( x - y\right) }{\Vert y \Vert^{c+s}}dy
\nonumber  \\
& \leq & \Lambda \int_{\left\lbrace \delta_{w}^{+} > 0 \right\rbrace } \dfrac{v^{-}\left( x + y\right) + v^{-}\left( x - y\right) }{\Vert y \Vert^{c+s}}dy.
\end{eqnarray}
Analogously, we get
\begin{eqnarray}
\label{Harnack Inequality Estimate 7}
I_{2} & = &  \Lambda \int_{\left\lbrace \delta_{v}^{-} > 0\right\rbrace \cap \left\lbrace \delta_{w}^{-} \neq \delta^{-}_{v} \right\rbrace } \dfrac{\delta_{v}^{-} - \delta_{w}^{-}}{\Vert y \Vert^{c+s}}dy \nonumber\\  
&  & + \Lambda \int_{\left\lbrace \delta_{v}^{-} = 0 \right\rbrace \cap \left\lbrace \delta_{w}^{-} \neq \delta^{-}_{v} \right\rbrace} \dfrac{v^{-}\left( x + y\right) + v^{-}\left( x - y\right) }{\Vert y \Vert^{c+s}}dy \nonumber \\ 
& \leq & \Lambda \int_{\left\lbrace \delta_{v}^{-} > 0\right\rbrace \cap \left\lbrace \delta_{w}^{-} \neq \delta^{-}_{v} \right\rbrace } \dfrac{- \delta_{v} - \delta_{v}^{-}}{\Vert y \Vert^{c+s}}dy.
\end{eqnarray}
We also have 
\begin{eqnarray}
\label{Harnack Inequality Estimate 8}
- \delta_{v}^{-} - \delta_{w}^{-} & = & 2v\left( x \right) - \left( v\left( x + y \right) + v\left( x - y \right)\right) - \delta_{w}^{-} \nonumber\\ 
& = & 2v\left( x \right) - \left[ \left( v^{+}\left( x + y \right) + v^{+}\left( x - y \right)\right) \right. \nonumber \\
& & \left. - \left( v^{-}\left( x + y \right) + v^{-}\left( x - y \right)\right) \right] \nonumber\\ 
& = & \left( -\delta_{w} - \delta_{w}^{-} \right) +  v^{-}\left( x + y \right) + v^{-}\left( x - y \right) \nonumber\\ 
& = & -\delta^{+}_{w} + v^{-}\left( x + y \right) + v^{-}\left( x - y \right).
\end{eqnarray}
Then, from \eqref{Harnack Inequality Estimate 8} and \eqref{Harnack Inequality Estimate 7}, we obtain
\begin{eqnarray}
\label{Harnack Inequality Estimate 9}
I_{2} & \leq &  - \Lambda \int_{\left\lbrace \delta_{v}^{-} > 0\right\rbrace \cap \left\lbrace \delta_{w}^{-} \neq \delta^{-}_{v} \right\rbrace } \dfrac{ \delta_{w}^{+}}{\Vert y \Vert^{c+s}} dy \nonumber \\  
&  & + \Lambda \int_{\left\lbrace \delta_{v}^{-} > 0\right\rbrace \cap \left\lbrace \delta_{w}^{-} \neq \delta^{-}_{v} \right\rbrace} \dfrac{v^{-}\left( x + y\right) + v^{-}\left( x - y\right) }{\Vert y \Vert^{c+s}}dy \nonumber\\ 
& \leq & \Lambda \int_{\left\lbrace \delta_{w}^{-} \geq 0 \right\rbrace} \dfrac{v^{-}\left( x + y\right) + v^{-}\left( x - y\right) }{\Vert y \Vert^{c+s}}dy.
\end{eqnarray}
Hence, using \eqref{Harnack Inequality Estimate 6}, \eqref{Harnack Inequality Estimate 9}, and changing variables, we find
\begin{eqnarray*}
\dfrac{ \mathcal{M}^{-}w\left( x \right) - \mathcal{M}^{-}v\left( x \right)}{q_{\max, s}} & \leq &  \Lambda \int_{\mathbb{R}^{n}} \dfrac{v^{-}\left( x + y\right) + v^{-}\left( x - y\right) }{\Vert y \Vert^{c+s}}dy \\ 
& = & -2 \Lambda \int_{\left\lbrace v\left( x + y \right) < 0 \right\rbrace } \dfrac{v\left( x + y\right)}{\Vert y \Vert^{c+s}}dy.
\end{eqnarray*}
Moreover, if $x \in B_{\frac{r\theta}{2}}\left( x_{0} \right) $, we have
$$\dfrac{\mathcal{M}^{-}w\left( x \right) - \mathcal{M}^{-}v\left( x \right)}{q_{\max, s}}  \leq   2 \Lambda \int_{\mathbb{R}^{n} \setminus B_{r \theta }\left( x_{0} - x \right)} \dfrac{-v\left( x + y\right)}{\Vert y \Vert^{c+s}}dy$$ 
$$ \leq  2 \Lambda \int_{\mathbb{R}^{n} \setminus B_{r \theta }\left( x_{0} - x \right)} \dfrac{\left( u\left( x + y\right) - \left( 1 - \frac{\theta}{2} \right)^{-\tau}u\left( x_{0} \right)\right)^{+}}{\Vert y \Vert^{c+s}}dy.$$
If $\iota > 0$ is the largest value such that $u\left( x \right) \geq \iota \left( 1 - \vert 4 x \vert^{2} \right)$, then there is a point $x_{1} \in B_{\frac{1}{4}}$ such that $u\left( x_{1} \right) = \left( 1 - \vert 4 x_{1} \vert^{2} \right)$. Moreover, since $u\left( 0 \right)\leq 1$, we get $\iota \leq 1$. Then, we have 
$$q_{\max, s} \int_{\mathbb{R}^{n}} \dfrac{\delta^{-} \left( u, x_{1}, y\right)}{\Vert y \Vert^{c+s}}dy \leq q_{\max, s} \int_{\mathbb{R}^{n}} \dfrac{\delta^{-} \left( \left( 1 - \vert 4 x \vert^{2} \right), x_{1}, y\right)}{\Vert y \Vert^{c+s}}dy \leq C,$$
where the constant $C>0$ is independent of $s$. Moreover, since $\mathcal{M}^{-}u\left( x_{1}\right) \leq 1 $, we find
$$q_{\max, s} \int_{\mathbb{R}^{n}} \dfrac{\delta^{+} \left( u, x_{1}, y\right)}{\Vert y \Vert^{c+s}}dy \leq  C.$$
Recall that $u\left( x_{1} - y \right)\geq 0$ and $u\left( x_{1} \right)\leq 1$. Thus, we obtain
$$q_{\max, s} \int_{\mathbb{R}^{n}} \dfrac{\left(  u\left( x_{1} + y \right) - 2 \right)^{+} }{\Vert y \Vert^{c+s}}dy \leq  C.$$
Since $t>0$ is large enough, we can suppose that $ u\left( x_{0} \right) > 2$. Let 
$$x \in E^{\max}_{  \frac{r\theta}{2n},1}\left( x_{0} \right) \subset B_{\frac{r \theta}{2n} }\left( x_{0} \right) \subset B_{\frac{r \theta}{2} }\left( x_{0} \right)$$ 
and 
$$y \in \mathbb{R}^{n} \setminus B_{r \theta }\left( x_{0} - x \right) \subset \mathbb{R}^{n} \setminus E^{\max}_{ r \theta,1} \left( x_{0} - x \right).$$
Then, we have the inequalities 
\begin{eqnarray*}
\Vert y + x + x_{1} \Vert  \leq  C \left(  \Vert y \Vert + \Vert x \Vert + \Vert x_{1} \Vert \right)  \leq  C \Vert y \Vert + 2C 
\end{eqnarray*}
and 
\begin{eqnarray*}
| y_{i}| & \geq & | \left( y - \left( x_{0} - x \right)\right)_{i} | - | \left( x_{0} - x \right)_{i} |\\ 
& \geq & \frac{\left( r\theta\right)^{\frac{b_{\max}}{b_{i}}}}{n^{1/2}} - \left( \frac{ r\theta}{2n}\right)^{\dfrac{b_{max}}{b_{i}}} \\
& \geq &\frac{\left( r\theta\right)^{\frac{b_{\max}}{b_{i}}}}{n} - \frac{\left( r\theta\right)^{\frac{b_{\max}}{b_{i}}}}{2n} \\
&\geq & \frac{\left( r\theta\right)^{\frac{b_{\max}}{b_{i}}}}{2n}.
\end{eqnarray*}    
Then, taking into account the obvious equalities
$$u\left( x + y\right) - \left( 1 - \frac{\theta}{2} \right)^{-\tau}u\left( x_{0} \right) =  u\left( x + x_{1} + y - x_{1} \right) - \left( 1 - \frac{\theta}{2} \right)^{-\tau}u\left( x_{0} \right),$$
and
$$ \dfrac{1}{\Vert y \Vert^{c+s}} = \dfrac {1}{\Vert y + x + x_{1} \Vert^{c+s}}  \dfrac{\Vert y + x + x_{1} \Vert^{c+s}}{\Vert y \Vert^{c+s}}, $$
we estimate
$$2 \Lambda \int_{\mathbb{R}^{n} \setminus E^{\max}_{  r \theta,1} \left( x_{0} - x \right)} \dfrac{\left( u\left( x + y\right) - \left( 1 - \frac{\theta}{2} \right)^{-\tau}u\left( x_{0} \right)\right)^{+}}{\Vert y \Vert^{c+s}}dy  \leq  C_{1}\left( \theta r \right) ^{-\frac{b_{\max}}{2} \left( c + s \right)}.$$
Thus, we have
$$\mathcal{M}^{-}w \leq C_{1}\left( \theta r \right) ^{-\frac{b_{\max}}{2} \left( c + s \right)} \quad \text{in} \ E^{\max}_{  \frac{r\theta}{2n},1}\left( x_{0} \right).$$
Applying Theorem \ref{Point Estimates} to $w$ in $E^{\max}_{  \frac{r\theta}{2n},1}\left( x_{0} \right) \subset B_{\frac{r \theta }{2}}\left( x_{0} - x \right)$ and using that 
$$w\left( x_{0}\right) =  \left( \left( 1 - \frac{\theta}{2} \right)^{-\tau} - 1 \right) u\left( x_{0} \right),$$ 
we get
\begin{eqnarray}
\label{Harnack Inequality Estimate 16}
& & \left | \left\lbrace u >  \dfrac{u\left( x_{0} \right)}{2} \right\rbrace \cap E^{\max}_{  \frac{r\theta}{2n},\frac{1}{2}} \right | \nonumber \\
& =  &   \left | \left\lbrace w > \left[ \left( 1 - \frac{\theta}{2} \right)^{-\tau} - \dfrac{1}{2} \right] u\left( x_{0} \right) \right\rbrace \cap E^{\max}_{  \frac{r\theta}{2n},\frac{1}{2}} \right | \nonumber\\ 
& \leq &  C \left | E^{\max}_{  \frac{r\theta}{2n},\frac{1}{2}} \right |  \left[ \left( \left( 1 - \frac{\theta}{2} \right)^{-\tau} - \dfrac{1}{2} \right) u\left( x_{0} \right) + C_{1} \left( r\theta \right)^{-\frac{b_{\max}}{2} \left( c + s \right) + s\frac{b_{\max}}{2}} \right]^{\varepsilon} \nonumber\\ 
& & \cdot  \left[ \left( \left( 1 - \frac{\theta}{2} \right)^{-\tau} - \dfrac{1}{2} \right) u\left( x_{0} \right)\right]^{-\varepsilon}\nonumber \\ 
& = & \nonumber C \left | E^{\max}_{  \frac{r\theta}{2n},\frac{1}{2}} \right |  \left[ \left( \left( 1 - \frac{\theta}{2} \right)^{-\tau} - \dfrac{1}{2} \right) u\left( x_{0} \right) + C_{1}\left( r\theta \right)^{- \frac{b_{\max}}{2}c } \right]^{\varepsilon}\nonumber \\ 
& & \cdot  \left[ \left( \left( 1 - \frac{\theta}{2} \right)^{-\tau} - \dfrac{1}{2} \right) u\left( x_{0} \right)\right]^{-\varepsilon}.
\end{eqnarray}
Thus, using \eqref{Harnack Inequality Estimate 16} and the elementary inequalities
$$ \left[ \left( \left( 1 - \frac{\theta}{2} \right)^{-\tau} - \dfrac{1}{2} \right) u\left( x_{0} \right) + C_{1}\left( r\theta \right)^{ -\frac{b_{\max}}{2}c} \right]^{\varepsilon}$$ 
$$\leq \left( \left( 1 - \frac{\theta}{2} \right)^{-\tau} - \dfrac{1}{2} \right)^{\varepsilon} u\left( x_{0} \right)^{\varepsilon} +  C_{1} \left( r\theta \right)^{- \frac{b_{\max}}{2}c \varepsilon}$$
and
$$\left( 1 - \frac{\theta}{2} \right)^{-\tau} - \frac{1}{2} = \left( 1 - \frac{\theta}{2} \right)^{-\frac{c b_{\max}}{2\varepsilon}} - \frac{1}{2} \geq \frac{1}{2},$$
for $\theta > 0$ sufficiently small, and yet
$$C_{3} \theta^{-\frac{c b_{\max}}{2}\varepsilon} r^{-\frac{c b_{\max}}{2}\varepsilon}u\left( x_{0} \right)^{-\varepsilon} \left(  \left( 1 - \frac{\theta}{2} \right)^{-\tau} - \frac{1}{2}\right)^{-\varepsilon} $$
$$ \leq  \nonumber C_{4} \theta^{-\frac{c b_{\max}}{2}\varepsilon} r^{-\frac{c b_{\max}}{2}\varepsilon}u\left( x_{0} \right)^{-\varepsilon} \leq C_{5}\theta^{-\frac{c b_{\max}}{2}\varepsilon} t^{-\varepsilon} d^{\left( 1 -  \varepsilon \right)\frac{cb_{\max}}{2}}  \leq  C_{6}\theta^{-\frac{c b_{\max}}{2}\varepsilon\varepsilon} t^{-\varepsilon}, $$
we obtain
$$\left | \left\lbrace u >  \dfrac{u\left( x_{0} \right)}{2} \right\rbrace \cap E^{\max}_{  \frac{r\theta}{2n},\frac{1}{2}} \right | \leq C \left | E^{\max}_{  \frac{r\theta}{2n},\frac{1}{2}} \right| \left[ \left( \left( 1 - \frac{\theta}{2} \right)^{-\tau} - 1 \right)^{\varepsilon} + \theta^{-\frac{c b_{\max}}{2}\varepsilon} t^{-\varepsilon} \right]. $$
Now we choose $\theta > 0$ sufficiently small such that
\begin{eqnarray*}
C \left |E^{\max}_{  \frac{r\theta}{2n},\frac{1}{2}} \right | \left[ \left( 1 - \frac{\theta}{2} \right)^{-\tau} - 1 \right]^{\varepsilon} & = & C \left |E^{\max}_{  \frac{r\theta}{2n},\frac{1}{2}} \right | \left[ \left( 1 - \frac{\theta}{2} \right)^{-\frac{c b_{\max}}{2\varepsilon}} - 1 \right]^{\varepsilon}\\
&  \leq & \dfrac{1}{4} \left |E^{\max}_{  \frac{r\theta}{2n},\frac{1}{2}} \right |.
\end{eqnarray*}
Having fixed $\theta > 0$ (independently of $t$), we take $t>0$ sufficiently large such that
$$C \left |E^{\max}_{\frac{r \theta }{2n},\frac{1}{2}} \right | \theta^{-\frac{c b_{\max}}{2}\varepsilon} t^{-\varepsilon} \leq \dfrac{1}{4} \left |E^{\max}_{\frac{r \theta }{2n},\frac{1}{2}} \right |.$$
Then, using \eqref{Harnack Inequality Estimate 16}, we find
$$ \left | \left\lbrace u >  \dfrac{u\left( x_{0} \right)}{2} \right\rbrace \cap E^{\max}_{  \frac{r\theta}{2n},\frac{1}{2}} \right |  \leq  \dfrac{1}{4} \left |E^{\max}_{  \frac{r\theta}{2n},\frac{1}{2}} \right |.$$
Hence, we have, for $t > 0$ large,
\begin{eqnarray*} 
\left | \left\lbrace u <  \dfrac{u\left( x_{0} \right)}{2} \right\rbrace \cap E^{\max}_{  \frac{r\theta}{2n},\frac{1}{2}} \right |  & \geq &  c \theta^ { \frac{b_{\max}}{2} c} \left |E^{\max}_{  r,1} \right | \\& \geq &  c_{2} \left |E^{\max}_{ r,1} \right |,
\end{eqnarray*}
which is a contradiction to \eqref{Harnack Inequality Estimate 1}.
\end{proof}

As a consequence of the Harnack inequality we obtain the $C^{\gamma}$ regularity.
  
\begin{theorem}[$C^{\gamma}$ estimates] \label{estimates}
\label{ estimates}
Let $u$ be a bounded function such that
$$ \mathcal{M}^{-} u \leq C_{0} \quad \text{and} \quad \mathcal{M}^{+} u \geq - C_{0}  \ \text{in}  \ B_{1}.$$
If $ 0< s_{0} < s < \frac{4}{b_{\max}}$, then there is a positive constant $0 < \gamma < 1$, that depends only $n$, $\lambda$, $\Lambda$, $b_{\min}$, $b_{\max}, s_{0}$, and $s$, such that $u \in C^{\gamma}\left( B_{1/2}\right)$ and 
$$| u |_{C^{\gamma}\left( B_{1/2}\right)} \leq C \left( \sup \limits_{\mathbb{R}^{n}}| u | + C_{0}  \right),$$
for some constant $C>0$.
\end{theorem}

The next result is a consequence of the arguments used in \cite{CS, CLU} and Theorem \ref{estimates}. As in \cite{CS, CLU}, if we suppose a modulus of continuity of $\mathcal{K}_{\alpha\beta}$ in measure, then as to make sure that faraway oscillations tend to cancel out, we obtain the interior $C^{1, \gamma}$ regularity for solutions of equation $\mathcal{I}u = 0$. 

\begin{theorem}[$C^{1,\gamma}$ estimates]
Suppose that $0 < s_{0} < s< \frac{4}{b_{\max}}$. There exists a constant $\tau_{0} > 0$, that depends only on $\lambda$, $\Lambda$, $n$, $b_{\min}$, $b_{\max}, s_{0}$ and $s$, such that
$$ \int_{\mathbb{R}^{n}\setminus B^{\tau_{0}}}\dfrac{| \mathcal{K}_{\alpha\beta}\left( y \right) - \mathcal{K}_{\alpha\beta}\left( y - h \right)|}{|h|} dy \leq C_{0}, \quad \ \text{whenever} \ |h| < \frac{\tau_{0}}{2} . $$
If $u$ is a bounded function satisfying $\mathcal{I}u=0$ in $B_{1}$, then there is a constant $0 < \gamma < 1$, that depends only $n$, $\lambda$, $\Lambda$, $b_{\min}$, $b_{\max}, s_{0}$ and $s$, such that $u \in C^{1,\gamma}\left( B_{1/2}\right)$ and 
$$| u |_{C^{1,\gamma}\left( B_{1/2}\right)} \leq C  \sup \limits_{\mathbb{R}^{n}}| u |,$$
for some constant $C=C\left( n, \lambda, \Lambda, b_{\min}, b_{\max}, s_{0}, s, C_{0} \right)>0$.
\end{theorem}

\begin{remark} As in \cite{CLU}, we can also obtain $C^{\gamma}$ and $C^{1, \gamma}$ estimates for truncated kernels, i.e., kernels that satisfy \eqref{Kernel cond 2} only in a neighborhood of the origin. Let $\mathfrak{L}$ be the class of operators $L_{\alpha\beta}$ such that the corresponding kernels $\mathcal{K}_{\alpha\beta}$ have the form
$$\mathcal{K}_{\alpha\beta}\left( y \right) = \mathcal{K}_{\alpha\beta, 1}\left( y \right) + \mathcal{K}_{\alpha\beta, 2}\left( y \right) \geq 0,$$
 where
$$\dfrac{\lambda q_{\max, s} }{\Vert y \Vert^{c+s}} \leq \mathcal{K}_{\alpha\beta, 1}\left( y \right) \leq      \dfrac{\Lambda q_{\max, s}}{\Vert y \Vert^{c+s}}$$
and $\mathcal{K}_{\alpha\beta, 2} \in L^{1}\left( \mathbb{R}^{n}\right)$ with $\Vert \mathcal{K}_{\alpha\beta, 2} \Vert_{L^{1}\left( \mathbb{R}^{n}\right)}\leq c_{1}$, for some constant $c_{1}>0$. The class $\mathfrak{L}$ is larger than $\mathfrak{L}_{0}$ but the extremal operators $\mathcal{M}_{\mathfrak{L}}^{-}$ and $\mathcal{M}_{\mathfrak{L}}^{+}$ are controlled by $\mathcal{M}^{+}$ and $\mathcal{M}^{-}$ plus the $L^{\infty}$ norm of $u$ (see Lemma 14.1 and Corollary 14.2 in \cite{CS, CLU}). Thus the interior $C^{\gamma}$ and $C^{1,\gamma}$ regularity follow.
\end{remark}

 \subsection*{Acknowledgments}

 \hspace{0.65cm} EBS and RAL thank the Analysis research group of UFC for fostering a pleasant and productive scientific atmosphere. EBS supported by CAPES-Brazil.

\end{document}